\documentclass{article}
\usepackage{amsmath}
\usepackage{amsthm}
\usepackage{amssymb}
\usepackage[all]{xy}
\usepackage{hyperref}
\hypersetup{colorlinks=true,linkcolor=blue,citecolor=blue,urlcolor=blue}
\usepackage[T1]{fontenc}
\usepackage[auth-sc]{authblk}
\title{On the mapping class groups of $\p^1$-bundles over $\mathbb{P}^2$}
\author{Tenglin Hu}
\affil{Academy of Mathematics and System Science, Chinese Academy of Science, Beijing 100190, China\quad
\texttt{hutenglin@amss.ac.cn}}	

\newtheorem{thm}{Theorem}[section]

\newtheorem{lemma}{Lemma}[section]
\newtheorem{cor}{Corollary}[section]
\newtheorem{prop}{Proposition}[section]

\theoremstyle{definition} 
\newtheorem{rmk}{Remark}
\newtheorem{nota}{Notation}
\newtheorem{eg}{Example}
\newtheorem*{ack}{Acknowledgements}

\newcommand{\m}{\mathrm{MCG}}
\newcommand{\z}{\mathbb Z}
\newcommand{\p}{\mathbb P}

\begin{document}
	
	\maketitle
	
	\begin{abstract}
		In this article we compute the mapping class group of the total space $S(\xi)$ of the sphere bundle of a 3-dimensional real vector bundle $\xi$ over the complex projective plane $\mathbb{P}^2$ with $\langle p_1(\xi), [\p^2] \rangle =8n+5$.  Examples of these manifolds include the Milnor hypersurface $M_1$ and its generalizations  $M_k=\{(z,z')\in \mathbb{P}^2\times\mathbb{P}^2 \ | \  \sum z^k_iz'_i=0\}$ with $k$ odd. 
	\end{abstract}
	
	\section{Introduction}		
The mapping class group $\m^+(M)$ of a smooth manifold $M$ is the group of isotopy classes of orientation preserving self diffeomorphisms of $M$. It is an important object that encodes fundamental information about the symmetries of $M$.  In this paper, we study the computation of $\m^+(M)$ for certain closed, simply-connected $6$-manifolds. 

Although closed, simply-connected $6$-manifolds with torsion-free homology groups are classified by their algebraic-topological invariants due to the work of Wall \cite{Wa66} and Jupp  \cite{Jup73}, there appears to be no systematic treatment of the computation of their mapping class groups. Particular interest has been devoted to manifolds admitting geometric structures. For instance, in \cite{KreSu} a class of $6$-manifolds with second Betti number $b_2=1$ is studied, including complete intersections of complex dimension $3$. 

The present paper considers a class of $6$-manifolds with $b_2=2$, which includes natural compact K\"ahler manifolds such as  projective bundles associated to rank-$2$ holomorphic vector bundles over the complex projective plane $\mathbb P^2$. More precisely, in this paper we study the mapping class group of the total spaces of $S^2$-bundles over $\mathbb P^2$. A priori, the structure group of such bundles is the orientation-preserving diffeomorphism group of $S^2$. It was shown by Smale  \cite{Sma59} that this group is homotopy equivalent to $SO(3)$. Consequently,  such bundles arise as the sphere bundles of $3$-dimensional real vector bundles over $\mathbb P^2$. 

Let $\mathrm{Vect}_{\mathbb R}^3(\mathbb P^2)$ denote the set of equivalence classes of $3$-dimensional real vector bundles over $\mathbb P^2$. In Proposition \ref{prop:real3} we show that the first Pontrjagin class induces a bijection 
$$p_1 \colon \mathrm{Vect}_{\mathbb R}^3(\mathbb P^2) \stackrel{\cong}{\longrightarrow}\{r\omega^2 \in H^4(\mathbb{P}^2):r\in 4\mathbb{Z}+\{0,1\}\},$$
where $\omega\in H^2(\mathbb P^2)$ is the hyperplane class.

Let $MCG(M)$ and $I(M)$ denote, respectively, the full mapping class group of $M$ (i.e., the group of isotopy classes of all self‑diffeomorphisms) and the Torelli group (i.e., the subgroup consisting of those mapping classes that induce the trivial action on cohomology). The index $[MCG(M)\colon MCG^+(M)]\leq2$ and equals 2 if and only if $M$ admits an orientation reversing diffeomorphism.
	
	\begin{thm}\label{Mth}
		Let $N_r=S(\xi_r)$ be the total space of the sphere bundle of a 3-dimensional real vector bundle $\xi_r$ over $\mathbb{P}^2$ with $p_1(\xi_r)=r \omega^2$. Then
		\begin{enumerate}
		\item There is a diagram with each row a short exact sequence 
		$$\xymatrix{
			1\ar[r]\ar[d]&I(N_r)\ar[r]\ar[d]^=&\m^+(N_r)\ar[r]\ar[d]&A^+_r\ar[r]\ar[d]&\ar[d]1\\
		    1\ar[r]&I(N_r)\ar[r]&\m(N_r)\ar[r]&A^+_r\oplus\z_2\ar[r]&1}$$
		where $A^+_r \cong \z_2$ if $r \ne -3$, and $A^+_r \cong S_3$, the symmetric group of $3$ elements, if $r=-3$;
		
		\item $I(N_r)$ is abelian.
		\item When $r=4l+1$, the action of $MCG(N_r)/I(N_r)$ on $I(N_r)$ induced by conjugation is given by : $$[h^{-1}\circ f\circ h]=deg(h)[f],\ \ \forall h\in\m(N_r),f\in I(N_r).$$ 
		\item When $r=4l+1$, there is a lower and upper bound of $|I(N_r)|$: $$\mathrm{gcd}(6,l^2+l)\mathrm{gcd}(28,l^2+l,4l+4)\leq |I(N_r)|\leq32\mathrm{gcd}(6,l^2+l)\mathrm{gcd}(28,l^2+l,4l+4).$$ 
		\item When $r=8n+5$, $$I(N_r)\cong \mathbb{Z}_{2\mathrm{gcd}(3,2n^2+3n+1)}\oplus \mathbb{Z}_{2\mathrm{gcd}(14,n+1)}.$$
		
		\end{enumerate}
	\end{thm}
	
\begin{rmk}
	In principle, similar lower and upper bounds for $|I(N_{4l})|$ can be obtained via the same Kreck-Stolz invariant, but the involved Atiyah-Hirzebruch spectral sequence calculation is rather tedious and, more importantly, it does not yield the exact isomorphism type of the Torelli group. Therefore, to avoid obscuring the precise computation achieved for $r=8n+5$, we leave the exact determination of $I(N_{4l})$ for future work.
\end{rmk}	
		
The projective unitary group $PU(2)$ is isomorphic to the orthogonal group $SO(3)$; consequently, $S^2$-bundles may be identified with $\p^1$-bundles. Thus, the manifolds considered in this paper include many complex K\"ahler manifolds, such as projective bundle associated to rank-$2$ holomorphic vector bundles over $\p^2$. 

Rank-$2$ complex vector bundles over $\p^2$ are classified by their Chern classes $c_1$ and $c_2$. For such a vector bundle $\gamma$ with Chern classes $c_1(\gamma)$ and $c_2(\gamma)$, the corresponding $3$-dimensional real vector bundle $\xi$ satisfies 
$$p_1(\xi) = c_1(\gamma)^2 - 4 c_2(\gamma).$$ 
(see Proposition \ref{prop:real3}). Consequently, Theorem \ref{Mth} also determines the mapping class group of the projective bundle $\p\gamma$ of rank-$2$ complex vector bundles $\gamma$ over $\p^2$, provided that $\langle c_1(\gamma)^2 - 4c_2(\gamma), [\p^2] \rangle = 8n+5$.

An interesting class of examples is given by the hypersurfaces $M_k \subset \p^2 \times \p^2$ defined by 
$$M_k=\{(z,z')\in \mathbb{P}^2\times\mathbb{P}^2 \ | \  \sum z^k_iz'_i=0\}.$$
Note that  $M_1=\{(z,z')\in \mathbb{P}^2\times\mathbb{P}^2 \ | \  \sum z_iz'_i=0\}$ is the Milnor hypersurface, a Fano 3-fold. The hypersurfaces $M_k$ can be identified with $N_{-3k^2}$ (see \S \ref{subsec:bundle}). We  therefore obtain the following corollary. 
	
	\begin{cor}\label{gmhs}
		Let $M_k=\{(z,z')\in \mathbb{P}^2\times\mathbb{P}^2 \ | \  \sum z^k_iz'_i=0\}$, then
		\begin{enumerate}
		\item there is a short exact sequence 
		$$ 1 \to I(M_k) \to \m^+(M_k) \to A \to 1,$$
		where $A \cong \z_2$ if $k \ne 1$, and $A \cong S_3$ if $k=1$;
		\item when $k=2l+1$, $I(M_k)$ is isomorphic to $\mathbb{Z}_6\oplus \mathbb{Z}_{\mathrm{gcd}(28,l^2+l)}$.    
		\end{enumerate}
	\end{cor}
		
To conclude this introduction, we briefly describe the strategy used in our computations. It combines methods from homotopy theory with classical and modified surgery theory. The determination of the action of $\m(M)$ on the cohomology groups relies on basic geometric considerations and constructions; this is carried out in \S \ref{sec:action}. 

The determination of the Torelli group $I(M)$ is more subtle. By classical surgery theory, for a simply-connected manifold $M$ of dimension at least $5$, there is a short exact sequence 
	$$ S(M \times I, \partial) \to I(M) \to I^h(M),$$
where $S(M \times I, \partial)$ denotes the relative smooth structure set of $M \times I$, and $I^h(M)$ denotes the homotopy Torelli group (see \S \ref{subsec:gen}) For the manifolds $N_r$, the image of $S(N_r \times I, \partial)$ in $I(N_r)$ is generated by a disk-supported diffeomorphism $g_1$. On the other hand, when $r \in 8 \mathbb Z +5$,
 homotopy-theoretical arguments show that there is a surjective homomorphism $\pi_6(N_r) \to I^h(N_r)$, whose image is generated by a certain Dehn twist $g_2$.  Consequently, in this case the Torelli group $I(N_r)$ is generated by $g_1$ and $g_2$ (see \S \ref{subsec:gen}). To determine the structure of $I(N_r)$ we make use of a Kreck-Stolz type invariant provided by modified surgery theory, which is an injective homomorphism from the relative Torelli group $I(N_r,D)$ to $\z^3$ modulo a lattice (see \S \ref{subsec:det}). 
 
 \begin{ack}
The author thanks Professor Yang Su for entrusting him with this project and for patiently guiding him through the research and writing process. The author is also grateful to Professor Stephen Theriault, who pointed out a key direction for proving Proposition \ref{ntb}, and Professor Matthias Kreck, who suggest the author to consider Proposition \ref{rconju}. This research was partially supported by NSFC 12471069 and Beijing Natural Science Foundation 1262024.
 \end{ack}
	
	\section{Basic properties of $S^2$-bundles over $\mathbb{P}^2$}
In this section we study the algebraic topology of $S^2$-bundles over $\p^2$. In \S \ref{subsec:bundle} we consider the classification of $S^2$-bundles over $\p^2$, and the algebraic topology of the total spaces. These facts are standard knowledge in the theory of vector bundles. For the convenience of the reader we give proofs of these facts. In \S \ref{subsec:hmtp} we analyze the homotopy groups of the total spaces. Some non-trivial arguments in homotopy theory are involved. 

	\subsection{Classification as bundles} \label{subsec:bundle}
For the sake of simplicity,	in the sequel we will use the same notation for (the equivalence class of) a vector bundle and (the homotopy class of) its classifying map from it base space to the classifying space.
	
	\begin{lemma}\label{lem:chern}
Let $c_1$ and $c_2$ denote the first and second Chern class, respectively,  of a complex vector bundle over $\p^2$. Then the map $(c_1,c_2) \colon  [\mathbb{P}^2, BU(2)]\to H^2(\mathbb{P}^2)\times H^4(\mathbb{P}^2)$ is a bijection.
	\end{lemma}
	\begin{proof}
The 4-th stage Postnikov tower of $BU(2)$ is given by taking the first two Chern classes of the universal bundle over $BU(2)$
$$(c_1,c_2) \colon BU(2) \to K(\mathbb{Z},2)\times K(\mathbb{Z},4).$$
Since $\mathbb{P}^2$ is a $4$-dimensional CW-complex, the statement of the lemma follows immediately.
\end{proof}
 
In the light of Lemma \ref{lem:chern}, we fix notations for the discussion in the remaining part of the paper.

\begin{nota} Let $\omega\in H^2(\mathbb{P}^2)$ be the first Chern class of the tautological bundle $O(-1)$. Denote by $\gamma_{k,l}$ the rank-$2$ complex vector bundle over $\p^2$ with total Chern class  $c(\gamma_{k,l})=1+k\omega+l\omega^2$ ($k, l \in \z$). 	
\end{nota}
	
	Let $PU(2)$ be the projective unitary group, $q \colon BU(2) \to BPU(2)$ be the map between classifying spaces induced by the quotient map $U(2) \to PU(2)$. 
	\begin{lemma}\label{lem:q}
		The map $q_* \colon [\mathbb{P}^2, BU(2)]\to[\mathbb{P}^2, BPU(2)]$ is surjective.
	\end{lemma}
	\begin{proof}
		The homotopy fiber of the map $q \colon BU(2) \to BPU(2)$ is $BS^1=K(\mathbb{Z},2)$. The obstruction groups to lifting a map $f \colon \mathbb{P}^2 \to BPU(2)$ along $q$ to a map $\mathbb{P}^2 \to BU(2)$ are $H^{n+1}(\mathbb{P}^2,\pi_n(K(\mathbb{Z},2)))$. These groups clearly vanish.

	\end{proof}
	
The representation  $SU(2)\to SO(3)$ 
	$$\begin{pmatrix} a&b\\\bar{b}&\bar{a} \end{pmatrix} \to \begin{pmatrix} |a|^2-|b|^2&2Im(a\bar{b})&2Re(a\bar{b})\\2Im(ab)&Re(a^2+b^2)&-Im(a^2-b^2) \\-2Re(ab)&Im(a^2+b^2)&Re(a^2-b^2)\end{pmatrix} $$
 induces an isomorphism (Lemma 1.23 \cite{Mar}) $\pi \colon  PU(2)=SU(2)/\{\pm1\}\to SO(3)$. 
 Let $\phi=\pi\circ q \colon U(2)\to SO(3)$ be the representation, then  $\phi$ associates a $3$-dimensional real bundle $\phi (\gamma)$ to a $2$-dimensional complex bundle $\gamma$. 	
	\begin{lemma}\label{lem:p1phi}
		$p_1(\phi(\gamma))=c_1^2(\gamma)-4c_2(\gamma)$ for any $2$-dimensional complex bundle $\gamma$.
	\end{lemma}
	\begin{proof}
		Let $\phi_{\mathbb{C}}:U(2)\to U(3)$ be the complexification of $\phi$. The Chern classes of the complex vector bundle $\phi(\gamma) \otimes \mathbb C$ can be computed  from the weights of the representation $\phi_{\mathbb C}$ by the Borel-Hirzebruch formula \cite{BrHi}. A direct calculation shows		
		$$\phi_{\mathbb{C}}\begin{pmatrix} e^{i\theta}&0\\0&e^{i\omega} \end{pmatrix} = \begin{pmatrix} 1&0&0\\0&cos(\theta-\omega)&-sin(\theta-\omega) \\0&sin(\theta-\omega)&cos(\theta-\omega)\end{pmatrix} $$
		$$=\frac{1}{2}\begin{pmatrix} 1&0&0\\0&1&-i \\0&1&i\end{pmatrix}\begin{pmatrix} 1&0&0\\0&e^{i(\theta-\omega)}&0 \\0&0&e^{i(\omega-\theta)}\end{pmatrix}\begin{pmatrix} 1&0&0\\0&1&1 \\0&i&-i\end{pmatrix}$$
So the weights of $\phi_{\mathbb C}$ are  $0$, $\theta-\omega$ and $\omega-\theta$. 
		
		Let $T^2\subset U(2)$ be the standard maximal tours and let $r_1,r_2\in H^2(BT^2)$ be the Chern roots of the universal $U(2)$ bundle. Then by the Borel-Hirzebruch formula, the total Chern class $c(\phi_{\mathbb{C}})=(1-r_1+r_2)(1-r_2+r_1)=1-(r_1+r_2)^2+4r_1r_2$. Thererfore $c_2(\phi(\gamma) \otimes \mathbb C)=4c_2(\xi)-c_1(\xi)^2$.  This proves the lemma. 
	\end{proof}

	\begin{lemma}\label{lem:p1}
Let $p_1$ denote the first Pontrjagin class of a real vector bundle over $\p^2$. Then the map $p_1 \colon  [\mathbb{P}^2, BSO(3)]\to H^4(\mathbb{P}^2)$ is injective.
	\end{lemma}
	\begin{proof}
		Suppose $f_1$ and $f_2\in [\mathbb{P}^2, BSO(3)]$ satisfy $p_1(f_1)=p_1(f_2)$. Let $w_2$ denote the second Stiefel-Whitney class. By the congruence relation $w_2^2 \equiv p_1 \pmod 2$, we have $w_2(f_1)=w_2(f_2)$. Since $3$-dimensional real vector bundles over $\p^1$ are classified by their $w_2$, we may assume (after a homotopy) that $f_1$ and $f_2$ coincide on $\mathbb{P}^1$. Therefore there exists some $\alpha\in \pi_4(BSO(3))$ such that under the action $\mu^* \colon \pi_4(BSO(3)) \times [\p^2, BSO(3)] \to [\p^2, BSO(3)]$ induced by the pinch map $\mu \colon \p^2 \to \p^2 \vee S^4$, one has $\mu^*(\alpha,f_1)=f_2$. It is ready to see $p_1(\mu^*(\alpha,f_1)) = p_1(f_1) + p_1(\alpha)$. Therefore the identity $p_1(f_1)=p_1(f_2)$ implies $p_1(\alpha)=0$, which further implies $\alpha=0$.
		\end{proof}

Now Lemmas \ref{lem:q}, \ref{lem:p1phi} and \ref{lem:p1} imply
	
	 \begin{prop}\label{prop:real3}
	 	The map 
		$$p_1 \colon  [\mathbb{P}^2, BSO(3)]\to \{r\omega^2\in H^4(\mathbb{P}^2) \ | \ r\in4\mathbb{Z}+\{0,1\}\}$$ 
	is a bijection and maps $\phi(\gamma_{k,l})$ to $(k^2-4l)\omega^2$.
	 \end{prop}
	 
\begin{nota} Denote by $\xi_r$ the $3$-dimensional real vector bundle over $\p^2$ with $p_1(\xi_r)=r\omega^2$, $r \in 4\mathbb Z + \{0,1\}$. Let $N_r=S(\xi_r)$ be the total space of the associated sphere bundle of $\xi_r$. 
\end{nota}
	
Let $p \colon \mathbb{P}\gamma_{k,l}\to \mathbb{P}^2$ be the projective bundle of $\gamma_{k,l}$. Then by Proposition \ref{prop:real3} we have $\mathbb{P}(\gamma_{0,-l}) = N_{4l}$, $\mathbb{P}(\gamma_{1,-l}) = N_{4l+1}$. 
 
 Let  $t\in H^2(\mathbb{P}\gamma_{k,l})$ be the first Chern class of the tautological line bundle $L$ over $\p(\gamma_{k,l})$, and $s=p^*x\in H^2(\mathbb{P}\gamma_{k,l})$. Then the cohomology ring $H^*(\mathbb{P}\gamma_{k,l})$ is expressed in terms of $s$, $t$, $k$ and $l$ by the Leray-Hirsch Theorem. We call the ordered pair $(s,t)$ the \emph{canonical generators} of $H^*(\mathbb{P}\gamma_{k,l})$. The characteristic classes of $\p(\gamma_{k,l})$ are computed from the identity of vector bundles $T\mathbb{P}\gamma_{k,l} \oplus \mathbb{C}\cong p^* T\mathbb{P}^2\oplus \mathrm{Hom}(L,p^*\gamma_{k,l})$. We summarize the results below.
	
	\begin{prop}\label{Hcp}
		$$H^*(\mathbb{P}\gamma_{k,l})=\mathbb{Z}[s,t]/(s^3,t^2-kts+ls^2);$$ 
		$$c_1(\mathbb{P}\gamma_{k,l})=(k-3)s-2t, \ \ p_1(\mathbb{P}\gamma_{k,l})=(k^2-4l+3)s^2.$$
		
	\end{prop}

\begin{eg}	\label{eg:1}
	 Let $\zeta_k$ be the complex vector bundle over $\p^2$ with total space 
	 $$E_r=\{(z,z')\in \mathbb{P}^2\times\mathbb{C}^3 \ | \  \sum z^k_iz'_i=0\}$$
 and bundle projection $\pi \colon E_r \to \p^2$,  $\pi(z,z')=z$.
	 One can identify $\zeta_k$ as conjugation of the orthogonal complement of canonical embedding $O(-k)\to\mathbb{C}^3$. So $\zeta_k\cong \gamma_{k,k^2}$ and $p_1(\phi( \zeta_k))=-3k^2$.  Notice that $\mathbb{P}(\zeta_1)$ is the Milnor hypersurface $\{(z,z')\in \mathbb{P}^2\times\mathbb{P}^2 \ | \  \sum z_iz'_i=0\}$. We call $M_k=\{(z,z')\in \mathbb{P}^2\times\mathbb{P}^2 \ | \  \sum z^k_iz'_i=0\}=\mathbb{P}(\zeta_k)$ a generalized Milnor hypersurface. One can easily show that the canonical generators of $M_k$ induced from the bundle $\zeta_k$ are $(\pi^* \omega,\pi'^* \omega)$. Here $\pi(z,z')=z, \pi'(z,z')=z'$. 
\end{eg}

	\subsection{homotopy groups of $N_r$}\label{subsec:hmtp}
In this subsection we study the homotopy groups of the total space $N_r=S(\xi_r)$ of the sphere bundle of the $3$-dimensional real vector bundle $\xi_r$ over $\p^2$ with $p_1(\xi_r)=r\omega^2$.  This information will be involved in the determination of the homomorphism from the Torelli group $I(N_r)$ to the homotopy Torelli group  $I^h(N_r)$. 

First of all, by elementary homotopy theory we have 
	$$\pi_1(N_r) =0, \ \pi_2(N_r)\cong H_2(N_r)\cong \mathbb{Z}\oplus\mathbb{Z}.$$ 
	Let $h \colon S^5\to\mathbb{P}^2$ be the Hopf map, $S(h^*\xi_r)$ be the sphere bundle of the pull-back vector bundle $h^*\xi_r$. Then $S(h^*\xi_r)$ is an $S^1$-bundle over $N_r$. The bundle map $S(h^*\xi_r) \to N_r$ induces an isomorphism $\pi_i(S(h^*\xi_r)) \to \pi_i(N_r)$ for $i>2$. If the pullback bundle  $h^* \xi_r$ is trivial, then $S(h^* \xi_r) \cong S^5 \times S^2$. Hence $\pi_i(N_r)\cong \pi_i(\mathbb{P}^2)\oplus \pi_i(S^2)$ for $i > 2$. In this subsection we determine for which $r$ the pull-back bundle $h^*\xi_r$ is nontrivial, and compute the homotopy groups of $N_r$ in this case. Note that there is only one isomorphism class of non-trivial $3$-dimensional real vector bundle over $S^5$, since $\pi_5(BSO(3))\cong \mathbb{Z}_2$.
		
	\begin{lemma}\label{lem:-3}
		$h^* \xi_{-3}$ is nontrivial and $h^* \xi_{1}$ is trivial. 
	\end{lemma}  
	\begin{proof}
For $h^* \xi_{-3}$, we have identifications $\xi_{-3}=\phi (\gamma_{-3,3})=\phi (T\mathbb{P}^2)$.  The homomorphism $\phi_* \colon \pi_5(BU(2))\to \pi_5(BSO(3))$ is an isomorphism. Therefore it suffices to show that $h^* T\mathbb{P}^2$ is nontrivial. This follows from the fact that $dh$ induces an isomorphism  of real bundles $TS^5\cong h^* T\mathbb{P}^2\oplus \mathbb{R}$, and $TS^5$ is nontrivial.

For $h^* \xi_{1}$, we have identifications $\xi_{1}=\phi (\gamma_{1,0})$. Note that $\gamma_{1,0}= O(-1)\oplus \mathbb{C}$, hence $h^*\gamma_{1,0}$ is a Whitney sum of two complex line bundles over $S^5$, which is  trivial. 
	\end{proof}
	
	\begin{prop}\label{ntb}
		$h^*\xi_r$ is nontrivial if and only if $r\in 8\mathbb{Z}+5$. 
	\end{prop}
	\begin{proof}
Notice that the pair $(\mathbb{P}^2\times S^4, \mathbb{P}^2\vee S^4)$ is $5$-connected,  thus we have a short exact sequence
$$0 \to \z \stackrel{u}{\longrightarrow} \pi_5(\mathbb{P}^2\vee S^4) \stackrel{i}{\longrightarrow} \pi_5(\mathbb{P}^2\times S^4) \to 0,$$
with a splitting 
$$(i_1+i_2) \circ (p_1,p_2) \colon \pi_5(\mathbb{P}^2\times S^4) \to 	\pi_5(\mathbb{P}^2)\oplus\pi_5(S^4) \to \pi_5(\mathbb{P}^2\vee S^4).$$
Here the maps $i$, $i_1$, $i_2$, $p_1$, $p_2$ are the ordinary inclusions and projections. The map $u$ sends $1 \in \z$ to the composite $S^5 \to S^2 \vee S^4 = \p^1 \vee S^4 \to \p^2 \vee S^4$, in which the first map is the Whitehead product and the last map is the canonical inclusion.

Consider the map $F=\mu\circ h: S^5\to \mathbb{P}^2\vee S^4$, where $\mu \colon \p^2 \to \p^2 \vee S^4$ is the pinch map. Regard $F$ as an element in $\pi_5(\p^2 \vee S^4)$, then by the above splitting exact sequence, $F$ is expressed as  $F=i_1(F_1)+i_2(F_2)+ku$. By the constructions, $F_1=p_1(F)=p_1\circ\mu\circ h=h$. By Thom-Pontrjagin construction, $F_2=p_2\circ\mu\circ h$ is characterized by the framed submanifold $h^{-1}([1:0:0])\subset S^5$. This framed submanifold can be view as $h(t,z_1,z_2)=[1:t^{-1}z_1:t^{-1}z_2]:S^1\times \mathbb{C}^2\to \mathbb{P}^2$ (Here we view $S^5\subset\mathbb{C}^3$ and idetify $S^1\times \mathbb{C}^2\subset \mathbb{C}^3$ as a framing of the normal bundle of $S^1\subset S^5$ in the obvious way), which is the ordinary normal framing of $S^1\subset S^5$ twisted  by twice of a generator of $\pi_1(U(2))$. Since $\pi_1(SO(4))\cong \mathbb{Z}_2$ the framing is trivial and $F_2=0$. Thus $F=i_1\circ h+ku$.

		Consider the induced map 
		$$F^* \colon [\p^2 \vee S^4 , BSO(3)]=[\mathbb{P}^2,BSO(3)]\times\pi_4(BSO(3))\to \pi_5(BSO(3))\cong \mathbb{Z}_2.$$ 
		By construction we have $F^*(\xi_r,\alpha)=h^*\xi_r+k\cdot(\xi_r, \alpha)\circ u$. The Whitehead product is bilinear, and $3$-dimensional real vector bundles over $\p^1$ are classified by the second Stiefel-Whitney class $w_2$; consequently $(\xi_r,\alpha)\circ u$ is bilinear in $(w_2(\xi_r),\alpha)$. 
		
The assumption $k \cdot (\xi_r,\alpha)\circ u=0$ for all $(\xi_r,\alpha)\in [\mathbb{P}^2,BSO(3)]\times\pi_4(BSO(3))$ will lead to a contradiction as shown in the following. Let $\beta\in \pi_4(BSO(3))$ be the generator with $\langle p_1(\beta),[S^4]\rangle=4$. Then by construction we have
$$F^*(\xi_1, \beta) = h^* \mu^*(\xi_1, \beta) = h^*\xi_{-3}$$
since $p_1(\xi_1)=\omega^2 = p_1(\xi_{-3}) + 4\omega^2$.
On the other hand, by the assumption, $F^*(\xi_1,\beta)=h^* \xi_1$. A contradiction (Lemma \ref{lem:-3}). 
		
	Now	since $k(\xi_r,\alpha)\circ u$ is not alway trivial and bilinear in $(w_2(\xi_r),\alpha)$, it must have the form $k(\xi_r, n\beta)\circ u=nw_2(\xi_r)$. So $h^*\xi_{4l}=F^* (\xi_0,l\beta)=h^* \xi_0$ is trivial, and $h^* \xi_{1+4l}=F^* (\xi_1,l\beta)=h^* \xi_1+l=l \in \pi_5(BSO(3)$. This completes the proof.
	\end{proof}

	\begin{lemma}\label{bh}
Let $v$ be the nontrivial $3$-dimensional real vector bundle over $S^5$, $\partial \colon \pi_{i+1}(S^5)\to \pi_i(S^2)$ be the boundary homomorphism in the long exact sequence of homotopy groups associated to the sphere bundle of $v$, $\Sigma:\pi_i(S^4) \to \pi_i(\Omega S^5)$ be the map induced by the adjoint of the identity map $S^5= \Sigma S^4 \to S^5$. Then the composite $\partial\circ\Sigma:\pi_i(S^4) \to \pi_i(\Omega S^5) = \pi_{i+1}(S^5) \to \pi_i(S^2)$ is identified with the homomorphism $\eta^2 \colon \pi_i(S^4) \to \pi_i(S^2)$ induced by the Hopf map $\eta \colon S^3 \to S^2$. 
		
	\end{lemma}
	\begin{proof}
		Let $Fr(v)$ be the oriented orthogonal frame bundle of $v$. Then the boundary homomorphism associated to the principal $SO(3)$-bundle $Fr(v)$, $\partial_{Fr} \colon \pi_5(S^5)\to \pi_4(SO(3))$ maps $[\mathrm{id}]$ to the clutching function $C_v$ of $v$. There is a commutative diagram
		$$\xymatrix{S^4\ar[d]\ar[dr]^{C_v}& & & \\
			\Omega S^5\ar[d]^=\ar[r] &SO(3)\ar[d]^{\mathrm{ev}}\ar[r]& Fr(v)\ar[d]^{\mathrm{ev}_{Fr}}\ar[r]& S^5\ar[d]^=\\
			\Omega S^5\ar[r] &S^2\ar[r]&S(v)\ar[r]&S^5}$$
where $\mathrm{ev}$ and $\mathrm{ev}_{Fr}$ are the corresponding evaluation maps.
The statement of the lemma follows from the fact that  $\mathrm{ev}\circ C_v=\eta^2$. 
	
	\end{proof}
	
	\begin{cor}\label{htg}
Let $N_r$ be the $S^2$-bundle over $ \mathbb{P}^2$ defined as above. Assume $r\in 8\mathbb{Z}+5$. Then
\begin{enumerate}
\item the inclusion of a fiber induces an isomorphism $\pi_3(S^2)\to\pi_3(N_r)$;	
\item$\pi_4(N_r)=0$; 
\item the bundle projection induces an injective homomorphism $\pi_5(N_r)\to \pi_5(\mathbb{P}^2)$ with image $2\pi_5(\mathbb{P}^2)$.
\item $\pi_6(S^2)\to\pi_6(N_r)$ is surjective and  $\pi_6(N_r)\cong \mathbb{Z}_6$.
\end{enumerate}
	\end{cor}
	
	\begin{proof}
By proposition \ref{ntb}, when $r \in 8\z +5$, the bundle $h^*\xi_r$ is non-trivial. By Lemma \ref{bh} and known facts about the homotopy groups of spheres \cite{Hat}, the long exact sequence of homotopy groups associate to the fiber bundle $S^2 \to N_r \to \p^2$ reduces to the following  long exact sequence
		$$\xymatrix{\z_2\{\eta^2\}\ar[r]^{\eta^2_*}&\pi_6(S^2)\ar[r]&\pi_6(N_r)\ar[r]&\z_2\{\eta\}\ar[r]^{\eta^2_*}&\z_2\{\eta^3\}\ar[r]&\\ \pi_5(N_r)\ar[r]&\z\{id\}\ar[r]^{\eta^2_*}&\z_2\{\eta^2\}\ar[r]&\pi_4(N_r)\ar[r]&0\ar[r]&\\
		\z\{\eta\}\ar[r]&\pi_3(N_r)\ar[r]&0&&&&}$$
		
Here the non-triviality of $\eta^4\in \pi_6(S^2)\cong \mathbb{Z}_{12}$ come from the non-triviality of $\eta^3\in \pi_6(S^3)$ although it is stably trivial.
		
	\end{proof}
	
\section{The action of $\m(N_r)$ on homology groups}\label{sec:action}
In this section we study the action of a diffeomorphism of $\p\gamma_{k,l}$ (or equivalently $N_r$)	on its cohomology ring. 	
Recall that $H^*(\mathbb{P}\gamma_{k,l})$ is generate by $H^2(\mathbb{P}\gamma_{k,l})=\mathbb{Z}\{s,t\}$ (Proposition \ref{Hcp}). Therefore the action of $\m(\p\gamma_{k,l})$ on $H^*(\mathbb{P}\gamma_{k,l})$ is  determined by the representation 
$$R_{k,l} \colon \m(\mathbb{P}\gamma_{k,l}) \to GL_2(\mathbb{Z}), \ \ (s,t)R_{k,l}(f) = f^*(s,t),$$
and the Torelli group $I(\mathbb{P}\gamma_{k,l})$ is equal to the kernel of $R$. 

Let $A_{k,l}$ and $A^+_{k,l}$ be the image of $R_{k,l}$ and the image of $MCG^+(\mathbb{P}\gamma_{k,l})$ under $R_{k,l}$. In this section we determine $A_{k,l}$ and $A^+_{k,l}$, which provides the short exact sequence in Theorem \ref{Mth}.

We first construct diffeomorphisms which acts non-trivially on cohomology. 		Let $g:\mathbb{P}^2\to \mathbb{P}^2$ be the complex conjugation map. Since $g^*$ maps the Chern classes of $\gamma_{k,l}$ to the Chern classes of $\gamma_{-k,l}$, Lemma \ref{lem:chern} implies that $g^*\gamma_{k,l}$ is isomorphic to  $\gamma_{-k,l}$. By Proposition \ref{prop:real3} the projective bundles $\mathbb{P}\gamma_{k,l}$ and $ \mathbb{P}\gamma_{-k,l}$ are isomorphic. Choosing a bundle isomorphism $f \colon \mathbb{P}\gamma_{k,l}\to \mathbb{P}\gamma_{-k,l}$ covering $\mathrm{id}$, we have the following diagram
		$$\xymatrix{\mathbb{P}\gamma_{k,l}\ar[r]^f\ar[d]&\mathbb{P}\gamma_{-k,l}\ar[r]^{\hat{g}}\ar[d]&\mathbb{P}\gamma_{k,l}\ar[d]\\
			\mathbb{P}^2\ar[r]^{\mathrm{id}}&\mathbb{P}^2\ar[r]^{g}&\mathbb{P}^2}$$
Let $f_1=\hat{g}\circ f \colon \mathbb{P}\gamma_{k,l} \to \mathbb{P}\gamma_{k,l}$.

	\begin{lemma}
		For the diffeomorphism $f_1 \colon \mathbb{P}\gamma_{k,l} \to \mathbb{P}\gamma_{k,l}$, we have
		$$R_{k,l}(f_1)=X_k=\begin{pmatrix}-1&-k\\0&1\end{pmatrix}.$$
	\end{lemma}
	\begin{proof}
			Let $\{s,t\}$ and  $\{s',t'\}$ be the canonical basis (provided by the Leray-Hirsch Theorem) of $H^2(\mathbb{P}\gamma_{k,l})$ and $H^2(\mathbb{P}\gamma_{-k,l})$, respectively. By construction, $f^\star (s')=s, \hat{g}^*(s)=-s' $. Since $\hat{g}$ is covered by a complex vector bundle isomorphism we have $\hat{g}^\star(t)=t'$. Since $f$ preserves the almost complex structures of the total spaces we have 
$$f^\star ((-k-3)s'-2t')=f^\star (c_1(T\mathbb{P}\gamma_{-k,l}))=c_1(T\mathbb{P}\gamma_{k,l})=(k-3)s-2t.$$
This shows that $f^*(t')=t-ks$ and hence $R(\hat{g}\circ f)=X_k$.

	\end{proof}
	
	By viewing $\mathbb{P}\gamma_{k,l}$ as a $S^2$-bundle, the antipodal bundle map gives a orientation reversing map $f_-$. 
	
		\begin{lemma}
		$R_{k,l}(f_1^{-1}\circ f_-)=-id\in A_{k,l}$.
	\end{lemma}
	\begin{proof}
		It suffices to prove $R_{k,l}(f_-)=-R_{k,l}(f_1)$. Assume $f_-^\star(s,t)=(as+bt,ns+mt)$, $f_-$ is a orientation reversing bundle map implies $a=1,b=0,m=-1$, then from
		$$0=f_-^*(t^2-kts+ls^2)-(t^2-kts+ls^2)=2(k-n)ts+n(n-k)s^2$$ 
		one has $n=k$.
	\end{proof}
	\begin{cor}\label{A+}
		$A_{k,l}=A_{k,l}^+\oplus\z_2$
	\end{cor}
	\begin{proof}
		By definition, there is a well defined map $r:A_{k,l}\to \{\pm 1\}$ taking $R_{k,l}(f)$ to $1$ if and only if $f$ is orientation preserving. By the above Lemma and Proposition \ref{Hcp}, $r(-id)=-1$. This implies the restriction of $r$ on $\{\pm I\}\subset Center(A_{k,l})$ is an isomorphism. Thus $A_{k,l}=Ker(r)\oplus Im(r)=A_{k,l}^+\oplus\z_2$.
	\end{proof}
	
	Identify $\mathbb{P}\gamma_{1,1}$ with the Milnor hypersurface (see Example \ref{eg:1})
$$M_1=\{(z,z')\in \mathbb{P}^2\times\mathbb{P}^2 \ | \  \sum z_iz'_i=0\}.$$ 
Define a diffeomorphism $f_2 \colon M_1\to M_1$, $f_2(z,z')=(z',z)$. Since the canonical basis $\{s, t\}$ of $H^2(\mathbb{P}\gamma_{1,1})$  equals to $\{\pi^* x, \pi'^* x\}$ we have
	
	$$R_{k,l}(f_2)=H=\begin{pmatrix}0&1\\1&0\end{pmatrix}.$$
	
	\begin{lemma}\label{A}
		There are generators:
		\begin{enumerate}
			\item $A^+_{k,l}=\{I,X_k\}\cong\z_2$, If $k^2-4l\neq -3$.
			\item $A^+_{1,1}=\{I,X_1,H,HX_1,X_1H,HX_1H\}\cong S_3$ the symmetric group.		
		\end{enumerate}	
	\end{lemma}
	\begin{proof}
		Throughout this proof we keep proposition \ref{Hcp} in mind.
		
		Suppose $f \colon \mathbb{P}\gamma_{k,l} \to \mathbb{P}\gamma_{k,l}$ be a diffeomorphism such that  $f^*(s,t)=(as+bt,ns+mt)$. If $k^2-4l\neq -3$, since $f$ preserves $p_1(N_r)=(r+3)s^2 \ne 0$, we have $f^*(s^2)=s^2$. Therefore the identity  $0=f^*(s^3) =s^2(as+bt)=bs^2t$ implies $b=0$, and $a = \pm 1$, $m = \pm 1$.
		
		Since $f$ is orientation-preserving, we have $s^2t=f^*(s^2t)=ms^2t$,  thus $m=1$.
		
		If $a=1$, then from 
		$$0=f^*(t^2-kts+ls^2)-(t^2-kts+ls^2)=2nts+(n^2-kn)s^2$$ 
		one has $n=0$ and $R_{k,l}(f)=I$.
		
		If $a=-1$, then from 
		$$0=f^*(t^2-kts+ls^2)-(t^2-kts+ls^2)=2(k+n)ts+(n^2+kn)s^2$$ 
		one has $n=-k$ and $R_{k,l}(f)=X_k$.
	
Now assume $(k,l)=(1,1)$, first notice that we have the relation $t^2 -ts +s^2=0 \in H^*(\mathbb{P}\gamma_{1,1})$. Thus from the identities $$(t^2-ts+s^2)s=0, \ t(t^2-ts+s^2)=0$$ 
we have $t^2s=ts^2$ and $t^3=0$. The identity $0=f^*(s^3)=(as+bt)^3=3ab(a+b)t^2s$ implies $a=0$, or $b=0$, or $a+b=0$.

If $b=0$, then by the same argument as in the case for $r \ne -3$ we have $R_{1,1}(f)= I$ or $X_1$.

If $a=0$, then $b = \pm 1$, $n= \pm 1$.  Since $f$ is orientation-preserving, $f^*(ts^2)=ts^2$, therefore $n=1$. 
\begin{enumerate}
\item If $b=1$, then the identity $0=f^*(t^2-ts+s^2)-(t^2-ts+s^2)=2mts+(m^2-m)t^2$ implies $m=1$.
\item If $b=-1$, then the identity  $0=f^*(t^2-ts+s^2)-(t^2-ts+s^2)=(m+1)ts+(m^2+m)t^2$ implies $m=-1$.
\end{enumerate}

If $a+b=0$,  then $(a,b)=(1,-1)$ or $(-1,1)$ (since the matrix $R_{1,1}(f)$ is unimodular). From the identity 
$$ts^2=f^*(ts^2)=(ns+mt)(as+bt)^2=(ns+mt)(t^2+2ts+s^2)=-(m+n)ts^2$$ 
we have $n+m=-1$.
\begin{enumerate}
\item If $a=1$, then the identity $0=f^*(t^2-ts+s^2)=-(n^2+n)ts+2nt^2$ implies $n=0$.
\item If $a=-1$, then the identity $0=f^*(t^2-ts+s^2)=-(n^2+3n+2)ts+(2n+2)t^2$ implies $n=-1$.
\end{enumerate}		
Summarizing these calculations, and notice that $R_{1,1}(f_1)=X_1$, $R_{1,1}(f_2)=H$, one may conclude that the image of $MCG(\mathbb{P}\gamma_{1,1})$ under $R_{1,1}$ is
$$\{I,X_1,H,HX_1,X_1H,HX_1H\}.$$ 
This group is isomorphic to the symmetric group $S_3$.

	\end{proof}
	
	Let $A^+_r:=MCG^+(N_r)/I(N_r)$, then Proposition \ref{prop:real3} implies $A^+_{k^2-4l}\cong A^+_{k,l}$, and (by the above Lemma \ref{A} and Corollary \ref{A+}) we have the first part of the main theorem:
	
	\begin{cor}
		There is a diagram with each row a short exact sequence 
		$$\xymatrix{
			1\ar[r]\ar[d]&I(N_r)\ar[r]\ar[d]^=&\m^+(N_r)\ar[r]\ar[d]&A^+_r\ar[r]\ar[d]&\ar[d]1\\
			1\ar[r]&I(N_r)\ar[r]&\m(N_r)\ar[r]&A^+_r\oplus\z_2\ar[r]&1}$$
		
		where $A_r \cong \z_2$ if $r \ne -3$; and $A_r \cong S_3$, the symmetric group of $3$ elements, if $r=-3$.
	\end{cor}

	\section{Computation of the Torelli group $I(N_r)$}
We have determined the action of $\m(N_r)$ on the cohomology ring in last section. To describe the structure of $\m(N_r)$ we need to determine the Torelli group $I(N_r)$.	
		
	\subsection{generators of $I(N_r)$}\label{subsec:gen}
	
A natural idea to study the mapping class group is to relate it to the group of homotopy equivalences. To be precise, let $M$ be a simply closed connected manifold with $\dim M>5$. Let $G(M)$, $\mathrm{Diff}(M)^b$ and $\mathbb{S}(M)$ be the topological monoid of self-homotopy equivalences, the blocked diffeomorphism group and the surgery space of $M$, respectively (see \cite{BeMa}).  Let $S(M)=\pi_0(\mathbb{S}(M))$ denote the structure set of $M$. Then there is a fibration (\cite{Wa98,BeMa})
	$$\mathrm{Diff}(M)^b\to G(M)\to \mathbb{S}(M).$$
	The right-most segment of the long exact sequence of homotopy groups associated to this fibration can be identify with the exact sequence
	$$\cdots \to S_\partial(M\times I)\to \m(M)\to\pi_0(G(M))\to S(M),$$
where $\pi_0(G (M))$ is the group of homotopy classes of homotopy equivalences relative.
This can also be proved directly by the $h$-cobordism theorem of Smale and the pseudo-isotopy theorem of Cerf. Let $I^h(M)$ be the subgroup of $\pi_0(G(M))$ consists of self homotopy equivalences that induce the identity on cohomology. Then the above exact sequence reduces to an exact sequence
	$$ \cdots \to S_\partial(M\times I)\to I(M)\to I^h(M)\to S(M).$$
		
	\begin{lemma}\label{t7}
		Let $\Theta_7\cong \mathbb{Z}_{28}$ be the Milnor-Kervaire group of oriented differential structures on the $7$-sphere. Then connected sum operation induces a surjective map $cs:\Theta_7\to S_\partial(N_r\times I)$. 
	\end{lemma}  
	\begin{proof}
Consider the smooth surgery exact sequence ( Theorem 11.22 of \cite{LuTi})
$$L_8(\z) \to S_\partial(N_r\times I) \to [\Sigma N_r, G/O] \to L_7(\z).$$
A direct calculation shows that the group $H^i(\Sigma N_r;\pi_i(G/O))$ vanishes for all $i$. Therefore  by obstruction theory, the homotopy set $[\Sigma N_r, G/O]$ consists of a single element. Therefore
the map $L_8(\z) \to S_\partial(N_r\times I)$ is surjective. By the construction, the Wall realization map $L_8(\mathbb Z) \to S_{\partial}(N_r \times I)$ factors through $L_8(\z) \to \Theta_7 \to S_\partial(N_r\times I)$.
	\end{proof}
		
	\begin{lemma}\label{Ih}
		Let $\mu_N:N_r\to N_r\vee S^6$ be the pinch map, then induced map
		$$\mu_N^* \colon \pi_6(N_r)\to I^h(N_r),\ \alpha \mapsto (\mathrm{id} \vee \alpha)\circ\mu_N$$ is surjective if $r\in8\z+5$. 
	\end{lemma}  
	\begin{proof}
By Corollary \ref{htg} the Hurewicz homomorphism $\pi_6(N_r) \to H_6(N_r)$ is trivial. It's ready to see that the composite $(\mathrm{id} \vee \alpha)\circ\mu_N \colon N \to N \vee S^6 \to N$ induces the identity on homology groups, hence is in $I^h(N_r)$.

Let $i_1 \colon S^2 \to N_r$ be the inclusion of a fiber and $i_2 \colon S^2 \to N_r$ be a section of the restricted bundle of $S(\xi_r)$ to $\p^1 \subset \p^2$. Then clearly $i_{1*}[S^2]$ and $i_{2*}[S^2]$ form a basis of $\pi_2(N_r)\cong H_2(N_r)$. Given $f \in I^h(N_r)$, we may assume its restriction on $i_1(S^2) \vee i_2(S^2)$ is the identity. By the data in given in Corollary \ref{htg}, there a unique  obstruction $\mathfrak o (f) \in H^6(N_r, S^2 \vee S^2; \pi_6(N_r))$ to $f$ being homotopic to the identity. The group $H^6(N_r, S^2 \vee S^2; \pi_6(N_r))$  is naturally isomorphic to $\pi_6(N_r)$. And by the construction of the obstruction class, if $\mathfrak o(f)$ corresponds to $\alpha \in \pi_6(N_r)$, then $f$ is homotopic to $(\mathrm{id} \vee \alpha) \circ \mu$. 
	\end{proof}
	
Let $\alpha \colon D^k \to SO(n+1)$ be a smooth map sending a collar neighborhood of the boundary to the identity. The \emph{Dehn twist} parametrized by $\alpha$ is a diffeomorphism
 $$De(\alpha) \colon S^n \times D^k \to S^n \times D^k, \ De(\alpha)(x,y)= (\alpha(y)x, y).$$	
This construction induces a homomorphism $De \colon \pi_k(SO(n+1))\to \m_\partial(S^n\times D^k)$, where $\m_\partial(S^n\times D^k)$ is the mapping class group relative to the boundary, i.e., diffeomorphisms and isotopies are assumed to be the identity in a collar neighborhood of the boundary. When $S^n \times D^k$ is embedded as a codimension $0$ submanifold of $M$, we extend $De(\alpha)$ by the identity on the complement to a diffeomorphism of $M$,  denoted by the same notation, and still call it a Dehn twist.
	
	In \cite[Theorem 7.9 and 7.12]{KreSu}, the group $\m_\partial(S^2\times D^4)$ is computed and a system of generator is given. Let $g_1$ be a diffeomorphism of $D^6$ relative to the boundary, representing a generator of  $ \m_\partial(D^6)\cong \Theta_7$, $t_1\in\pi_3(SO(4))$ be a generator of $\mathrm{Im}\, \pi_3(SO(3))\cong\mathbb{Z}$, $t_2$ be a generator of $\pi_4(SO(3))\cong \mathbb{Z}_2$ and $i_\eta:S^3\times D^3\to S^2\times D^4$ be an embedding representing the Hopf map $\eta\in\pi_3(S^2)$. Fix an embedded $6$-disk in $S^2 \times D^4$ and extend $g_1$ to the complement by the identity; let $g_2$ and $g_3$ be the Dehn twists in $i_{\eta}(S^3 \times D^3)$ with parameters $t_1$ and $t_2$, respectively. Then	$$\m_\partial(S^2\times D^4)=\mathbb{Z}_{28}\{g_1\}\oplus\mathbb{Z}_{12}\{g_2\}\oplus\mathbb{Z}_2\{g_3\}$$
For any given embedding of $S^2 \times D^4$ in a $6$-manifold, we extend these diffeomorphisms by the identity on the complement and still denote them by the same notation. 		
	\begin{prop}\label{sur}
		Let $i \colon S^2\times D^4\to N_r$ be a local trivialization of the $S^2$-bundle. Then $I^h(N_r)$ is generate by the Dehn twist $g_2$ and $I(N_r)$ is generate by the disk-supported diffeomorphism $g_1$ and the Dehn twist $g_2$ if $r\in8\z+5$. 
	\end{prop}
	\begin{proof}
Choose a localization $i \colon S^2\times D^4\to N_r$ such that $i(S^2\times D^4)$ and $S^2\vee S^2$ (defined in the proof of Lemma \ref{Ih}) are disjoint. The diffeomorphisms $g_1$ and $g_2$, which are extended by the identity on the complement of $i(S^2 \times D^4)$, clearly induce the identity on the cohomology groups and thus (their isotopy classes) are in $I(N_r)$.    
		
		By Lemma \ref{t7} and the exactness of $S_\partial(N_r\times I)\to I(N_r)\to I^h(N_r)$, the kernel of the forgetful homomorphism $I(N_r) \to I^h(N_r)$ is generated by $g_1$. Thus we only need to show that the image of $g_2$ is a generator of $I^h(N_r)$. This will be done by showing that the image of $g_2$ generates  the image of $\mu_N^*$ (see Lemma \ref{Ih}).
		
		Let $\mu_3:S^3\times D^3 \to (S^3\times D^3)\vee S^6$ be the pinch map. It induces a map
		$$\mu_3^* \colon \pi_6(S^3)\to \pi_0(G_\partial(S^3\times D^3)), \ \alpha \mapsto (id\vee\alpha)\circ\mu_3.$$
where we represent $\alpha \in \pi_6(S^3)$ by a map $S^6 \to S^3 \times \{0\}$. 
View $S^3\times D^3$ as a submanifold of $N_r$ through the embedding $i\circ i_\eta \colon S^3 \times D^3 \to S^2 \times D^4 \to N_r$, extending the map $\mu^*(\alpha)$ by the identity on the complement we have a homotopy equivalence $\mu^*(\alpha)$ of $N_r$. For $\alpha\in\pi_6(S^3)$, consider the composite $i \circ \eta \circ \alpha \colon S^6 \to S^3 \to S^2 \to N_r$. From the construction of the maps one readily verifies that $\mu_3^*(\alpha)=\mu_N^*(i \circ\eta \circ \alpha)\in I^h(N_r)$, i.e., there is a commutative diagram
$$\xymatrix{\pi_6(S^3) \ar[r]^{(i \circ \eta)_*} \ar[d]_{\mu_3^*} & \pi_6(N_r) \ar[d]^{\mu_N^*} \\
 \pi_0(G_\partial(S^3\times D^3)) \ar[r]^{\ \ \ \ \ \Phi} & I^h(N_r)}$$
 where the map $\Phi$ at the bottom is the extension by the identity on the complement of $i\circ i_{\eta}(S^3 \times D^3)$.
		
It can be verified by elementary homotopy theory (compare \cite[p.98]{BeMa}) that the homotopy Dehn twist map $De^h:\Omega^3 G(S^3)\to G_\partial(S^3\times D^3)$ defined by $De^h(\alpha)(x,y)=(\alpha(y)x,y)$ 
induces an isomorphism 
$\pi_0(\Omega^3G(S^3)) \to \pi_0(G_\partial(S^3\times D^3))$.  Under the natural isomorphisms
$$\pi_0(\Omega^3G(S^3)) \cong \pi_3(G(S^3), \mathrm{id}) \cong \pi_3((\Omega^3S^3)_{\mathrm{id}}) \cong \pi_6(S^3),$$
where $(\Omega^3S^3)_{\mathrm{id}}$ is the connected component containing $\mathrm{id}_{S^3}$, this map can be identified with $\mu_3^* \colon \pi_6(S^3) \to \pi_0(G_\partial(S^3\times D^3))$. The following commutative diagram
$$\xymatrix{\pi_3(SO(3))\ar[d]\ar[rr]^J&&\pi_3(\Omega^3S^3)\ar[d]_{\mu_3^*}\\
			\pi_3(SO(4))\ar[r]^{De \ \ \ \ }&\m_{\partial}(S^3\times D^3)\ar[r]&\pi_0(G_\partial(S^3\times D^3))}$$ 
shows that the image of $g_2$ in $I^h(N_r)$ equals $\Phi(\mu_3^*(J(t_1)))$.
The $J$-homomorphism $J \colon \pi_3(SO(3))\to\pi_6(S^3)$ is surjective. Therefore by Corollary \ref{htg} and Lemma \ref{Ih},  the image of $g_2$ generates $I^h(N_r)$.   This finishes the proof of this proposition.	
		
	\end{proof}
	\subsection{A brief introduction to Kreck's result}\label{bri}

In this subsection we give a brief introduction of some results in \cite{Kre18} and explain how these results work for our situation.

Let $M$ be a simply connected $6$-manifold with $H^2(M)$ torsionfree, $H^3(M)=0$ and $H^2(M)\cup H^2(M)=H^4(M)$. A polarization of $M$ is a basis $b=(b_1,b_2,..,b_m)$ of $H^2(M)$ such that  $w_2(M)\equiv b_1\pmod 2$ if $M$ is non-spin.

Let $B(m)=K(\z^m,2)\times BSpin$, define $p_0,p_1:B\to BSO$ be 
$$p_0=\oplus\circ(\star, \pi):K(\z^m,2)\times BSpin \to BSO\times BSO\to BSO$$ 
$$p_1=\oplus\circ(pr_1^\star O(-1), \pi):K(\z^m,2)\times BSpin\to BSO\times BSO\to BSO$$ 
here $\pi \colon BSpin\to BO$ is the canonical projection, $pr_1 \colon \z^m\to \z$ is the projection to the first factor, $O(-1)$ is the tautological line bundle over $CP^\infty=K(\z,2)$ and $\oplus :BSO\times BSO\to BSO$ is the $H$-space structure of $BSO$. 

Then we have pullback squares:
$$\xymatrix{
	B(m)\ar[r]^{pr}\ar[d]^{p_0}&K(\z^m,2)\ar[d]^0&&B(m)\ar[r]^{pr}\ar[d]^{p_1}&K(\z^m,2)\ar[d]^{pr_1(mod2)}\\
	BSO\ar[r]^{w_2}&K(\z_2,2)&&BSO\ar[r]^{w_2}&K(\z_2,2)
}$$
here $w_2:BSO\to K(\z_2,2)$ is the universal Stiefel-Whitney class.

Let $b=(b_1,b_2,..,b_m)$ be a polarization of $M$. If $M$ is spin, $-[TM]\in KSO(M)=[M,BSO]$ admits a lifting $g_{Spin}$ to $BSpin$. Let $(B,p):=(B(m),p_0)$ and $\hat{g}:=(b,g_{Spin}) \colon M \to B$. If $M$ is non-spin, view $b_1$ as a complex line bundle over $M$, then $-[TM\oplus b_1]\in KSO(M)=[M,BSO]$ admits a lifting $g_{Spin^c}$ to $BSpin$. Let $(B,p):=(B(m),p_1)$ and $\hat{g}:=(b,g_{Spin^c}) \colon M \to B$. Since $p\circ\hat{g}=-[TM]\in KSO(M)=[M,BSO]$ we may assume $\hat{g}$ is a normal $(B,p)$-structure ( i.e. a lift of the normal Gauss map $M\to BSO$ ) of $M$ and call it the normal $(B,p)$-structure induced by polarization $b$. In fact $p$ is the normal $2$-type of $M$ ( i.e. the $2$-stage Moore-Postnikov tower of the normal Gauss map).

\cite{Kre18} proved that

\begin{thm}\label{Kre18}(M.Kreck 2018)
	
	Suppose $b$ is a polarization of $M$, $\hat{g}:\colon M \to B$ is the normal $(B,p)$-structure induced by $b$, and $[f]\in I(M)$. Then $[f]=0$ if and only if:
	\begin{enumerate}
		\item There exist a normal $(B,p)$-structure $\hat{G}$ on the mapping tours $T_f$ extending $\hat{g}$.
		\item There exist a $(B,p)$-coboundary $(W,\bar{G})$ of $(T_f,\hat{G})$ such that $sign(W)=0$.
		\item $\langle\bar{G}^\star\alpha\cup\bar{G}^\star\beta,[W,\partial W]\rangle=0$, $\forall$ $\alpha,\beta\in Ker(\hat{G}^\star:H^4(B)\to H^4(T_f))$.
	\end{enumerate}
\end{thm}

\begin{nota}
	The notation $\langle\bar{G}^\star\alpha\cup\bar{G}^\star\beta,[W,\partial W]\rangle$ means $\langle\bar{\alpha}\cup\bar{\beta},[W,\partial W]\rangle$, where $\bar{\alpha}, \bar{\beta}\in H^4(W,\partial W)$ are preimage of $\bar{G}^*\alpha, \bar{G}^*\beta\in H^4(W)$, respectively, and this is well defined.
\end{nota}

\begin{rmk}
	The original theorem uses $\mathbb{Q}$ coefficient, but in these cases $H^4(B)$ and $H^4(T_f)$ are free Abelian groups. 
\end{rmk}

\begin{rmk}\label{rmk}
	In fact, the prove of the theorem in \cite{Kre18} shows that under the above condition $(W,\bar{G})$ is $(B,p)$-bordant to $(W',\bar{G}')$ relative to boundary such that $(W';M\times [0,\frac{1}{2}],M\times [\frac{1}{2},1])$ is an $s$-cobordism. 
\end{rmk}

\subsection{Generalized Kreck-Stolz invariant}
Assume the same as last subsection, similar to \cite{KreSu} we define the generalized Kreck-Stolz invariant on $I(M,D)$, and investigate the conjugation action of $FMCG(M,D)$ on $I(M,D)$.

Let $K:=Ker(\hat{g}^\star:H^4(B)\to H^4(M))$. For $(W,\bar{G})$ a $8$-dimensional $(B,p)$-manifold such that the restriction of $\bar{G}^\star K$ on $\partial W$ are $\{0\}$. Define $KS(W,\bar{G})\in Hom(Sym^2K,\z)$ to be:
$$KS(W,\bar{G})([\alpha\otimes \beta])=\langle\bar{G}^\star\alpha\cup\bar{G}^\star\beta,[W,\partial W]\rangle$$

Fix an embedded $6$-disk $D \subset M$, let $FMCG(M,D)$ (resp.$I(M,D)$) be the group of $D$-fixing isotopy classes of $D$-fixing diffeomorphisms (resp. which induce the identity on cohomology). By the disk theorem of Palais and the isotopy extension theorem, the forgetful homomorphism $I(M,D)\to I(M)$ is surjective. By a theorem from (\cite{Wa63} p.265), $\mathrm{Ker}(I(M,D)\to I(M))$ is trivial or isomorphic to $\mathbb{Z}_2$. 

Given a diffeomorphism $f \colon M \to M$ with $f|_D = \mathrm{id}$, let $T_f$ be the mapping torus of $f$ with the preferred embedding $S^1\times D \subset T_f$. Assume that the action of $f$ on the cohomology groups is trivial, by Wang's sequence $H^{2i}(T_f)=H^{2i}(M)=H^{2i+1}(T_f)$. Let $g_D:S^1\times D\to B$ be the normal $(B,p)$ structure induced by $D^2\times D$. Then we have:

$$\xymatrix{
	&hofib(p)\ar[d]\ar[r]^=&hofib(p)\ar[d]\\
	M\cup S^1\times D\ar[r]^{\ \ \ \ \ \hat{g}\cup g_D}\ar[d]&B\ar[r]^{pr}\ar[d]^{p}&K(\z^m,2)\ar[d]^{0\ or\ pr_1(mod2)}\\
	T_f\ar[r]^{-[T_{T_f}]}\ar[rru]^{b}&BSO\ar[r]^{w_2}&K(\z_2,2)
}$$

It's ready to see by obstruction theory that there is a normal $B$-structure $\hat{G}$ on $T_f$, extending $\hat{g}\cup g_D$. It is shown in \cite[Theorem 6]{Kre18} that the $7$-dimensional $(B,p)$-bordism group $\Omega_7(B,p)$ is trivial, therefore $(T_f,\hat{G})$ admits a coboundary $(W,\bar{G})$ whose signature $\mathrm{sign}(W)$  is $0$ ( after connecting sum with some $\mathbb{H}P^2$ ). Since $H^4(T_f)=H^4(M)$ the restriction of $\bar{G}^*K$ on $\partial W= T_f$ vanish and the $KS$-invariant is defined by

$$KS \colon I(M,D)\to Hom(Sym^2K,\z)/L$$
$$KS(f):=KS(W,\bar{G})$$
$$L:=\{KS(X):[X]\in \Omega_8(B,p), Sign(X)=0\}$$
	\begin{prop}\label{KS}
	$KS \colon I(M,D)\to Hom(Sym^2K,\z)/L$ is a well defined injective homomorphism, and thus $I(M)$, $I(M,D)$ are abelian. 
\end{prop}
\begin{proof}
	This proof is along the same line of the proof of Proposition 6.2 of \cite{KreSu}, with more details to be checked. 
	
	To prove $KS$ is well defined, suppose $(W_1,\bar{G}_1),(W_2,\bar{G}_2)$ are two $B$ coboundary of $(T_f,\hat{G})$ and $\alpha_1,\alpha_2\in K$. Chose embeddings $i_j: W_j\to \mathbb{R}^N\times (-1)^j[0,+\infty)$ good enough such that $i_j(W_j)\cap\mathbb{R}^N=T_f$. Define $W:=i_1(W_1)\cup i_2(W_2)\cong W_1\cup_\partial W_2$ then $(W,G_w:=\bar{G}_1\cup \bar{G}_2)$ is a well defined element in $\Omega_8(B,p)$, by \cite{Wa69} $sign(W)=0$. Consider the following diagram
	$$\xymatrix{\oplus_j H^4(W_j,T_f)\ar[r]&\oplus_j H^4(W_j)\ar[r]&\oplus_j H^4(T_f)\\
		H^4(W,T_f)\ar[r]\ar[u]^\cong &H^4(W)\ar[r]\ar[u]&H^4(T_f)\ar[u]^{injective}
	}$$
	
	This implies the preimage $\bar{\alpha}_i$ of $G_w^*\alpha_i\in H^4(W)$ in $H^4(W,T_f)\cong\oplus_j H^4(W_j,T_f)$ is a preimage of $(\bar{G}^\star_1\alpha_i, \bar{G}^\star_2\alpha_i)$ and, $$KS(W_1,\bar{G}_1)(\alpha_1\otimes\alpha_2)-KS(W_2,\bar{G}_2)(\alpha_1\otimes\alpha_2)$$
	$$=\langle\bar{G}_1^\star\alpha_1\cup\bar{G}_1^\star\alpha_2,[W_1,\partial W_1]\rangle-\langle\bar{G}_2^\star\alpha_1\cup\bar{G}_2^\star\alpha_2,[W_2,\partial W_2]\rangle$$
	$$=\langle\bar{\alpha}_1\cup\bar{\alpha}_2,[W,T_f]\rangle$$
	$$=KS(W,G_w)(\alpha_1\otimes\alpha_2)\in L$$ 
	
	To prove $KS$ is a homomorphism, suppose $f_3,f_4\in I(M,D)$ and $(W_j,\bar{G}_j)$ is a $B$ coboundary of $(T_{f_j},\hat{G})$. Identify $(n,t)\in N_r\times [\frac{1}{3},\frac{2}{3}]\subset T_{f_3}$ and $(n,1-t)\in N_r\times [\frac{1}{3},\frac{2}{3}]\subset T_{f_4}$, define $W':=W_3\cup_{N_r\times [\frac{1}{3},\frac{2}{3}]}W_4$, then $\partial W'=T_{f_3\circ f_4}$. By the $M-V$ sequence of $(W';W_3,W_4)$ there are unique elements $b'_i\in [W',K(\z,2)]$ extending the cohomology factors of $\bar{G}_j$ and
	$$\xymatrix{
		&hofib(p)\ar[d]\ar[r]^=&hofib(p)\ar[d]\\
		W_3\sqcup W_4\ar[r]^{\ \ \bar{G}_3\sqcup \bar{G}_4}\ar[d]&B\ar[r]^{pr}\ar[d]^{p}&K(\z^m,2)\ar[d]^{0\ or\ pr_1(mod2)}\\
		W'\ar[r]^{-[TW']}\ar[rru]^{b'_\star\ \ \ \ }&BSO\ar[r]^{w_2}&K(\z_2,2)
	}$$
	By obstruction theory we can define a $(B,p)$-normal structure $G'_w$ on $W'$ extending $\bar{G}_j$. Consider the following diagram
	
	$$\xymatrix{\oplus_j H^4(W_j,T_{f_j})\ar[r]&\oplus_j H^4(W_j)\ar[r]&\oplus_j H^4(T_{f_j})\\
		H^4(W',T_{f_3}\cup T_{f_4})\ar[r]\ar[u]^\cong \ar[d]&H^4(W')\ar[r]\ar[u]\ar[d]&H^4(T_{f_3}\cup T_{f_4})\ar[u]^{injective}\ar[d]\\
		H^4(W',T_{f_3\circ f_4})\ar[r]&H^4(W')\ar[r]&H^4(T_{f_3\circ f_4})\\
	}$$
	Arguments analog to the proof that $KS$ is well defined, one can show $KS$ is a homomorphism through this diagram. Note that we need $H^4(T_{f_3}\cup T_{f_4})\to \oplus_j H^4(T_{f_j}))$ being injective since we need a preimage of $G'_w\alpha\in H^4(W')$ in $H^4(W',T_{f_3}\cup T_{f_4})$. 
	
	If $M$ is non-spin, by \ref{Kre18} and the following lemma, $KS$ is injective.
	
	If $M$ is spin, by remark\ref{rmk}, $KS(f)=0$ implies a $(B,p)$-coboundary $(W'',\hat{G})$ of $(T_f,\hat{G})$ such that $(W'';M\times [0,\frac{1}{2}],M\times [\frac{1}{2},1])$ is
	an $s$-cobordism. Since the spin structure of $W''$ is compatible to $S^1\times D$ we may extend $S^1\times D$ to be $D^2\times D\subset W$ and define a relative $h$-cobordism. Hence by the peudo-isotopy theorem of Cerf, $f=id\in I(M, D)$. 
	
\end{proof}

\begin{lemma}
	If $X$ is a simply connected non-spin closed manifold and $dimX\leq5$, then the forgetful map $MCG(X,D)\to MCG(X)$ is isomorphism.
\end{lemma}
\begin{proof}
	Surjective follows from the disk theorem of Palais and the isotopy extension theorem, only need to show it is injective. Suppose $f\in Ker(MCG(X,D)\to MCG(X))$ and $i:S^1\times D\to T_f$ the preferred embedding. Since $f$ is isotopic to $id$ we have $T_f\cong X\times S^1=\partial (X\times D^2)$ and $i(S^1\times\{0\})\in X\times D^2$ extend to an embedding $j:D^2\to X\times D^2$ compatible with boundary. Since $D^2$ is contractible the normal bundle of $j(D^2)\in X\times D^2$ is trivial and its trivialization induce a framing of the normal bundle of $i(S^1\times\{0\})\in T_f$. If this framing is not compatible with $i(S^1\times D)$ we may alter $j$ by   taking connected sum with an embedding $S^2$ with non-trivial normal bundle since $X$ is non-spin. Thus $S^1\times D$ extended to $D^2\times D\subset X\times D^2$ and define a relative $h$-cobordism. Hence by the peudo-isotopy theorem of Cerf, $f=id\in MCG(X,D)$. 
\end{proof}

\begin{rmk}\label{S2S4}
	For $N$ a simply connected $6$-manifold with $H^2(M)$ torsionfree and $H^3(M)=0$ (possibly not satisfying $H^2(M)\cup H^2(M)=H^4(M)$), $(B,p)$ and $KS$ are well defined in the same way, but (in general) $KS$ may no longer being injective.
\end{rmk}

For a fixed base $\sigma=(\sigma_1,\sigma_2,..,\sigma_N)$ of $Sym^2K$,  we identify $KS$ as the $KS$-invariant induced by $s$:
$$KS_\sigma=(S_1,S_2,..,S_N) \colon I(N_r,D)\to \mathbb{Z}^N/L_\sigma$$
$$S_i(f)=KS(f)(\sigma_i)$$
$$L_\sigma=\{KS(X)(\sigma_1,\sigma_2,..,\sigma_N):X\in\Omega_8(B,p), Sign(X)=0\}$$
	
Suppose $h\in MCG(M,D)$, then $\hat{g}\circ h:M\to B$ is another normal $2$-type of $M$. By \cite{Kre99} there exist a fiber homotopy equivalence $\hat{h}:B\to B$ covering $p:B\to BSO$ such that $\hat{g}=\hat{h}\circ \hat{g}\circ h$. Thus $\hat{h}$ induce automorphisms of $K$ and $L$. Let $f\in I(M,D)$ and $(T_f,\hat{G})$, $(W,\bar{G})$ be the $(B,p)$ manifold needed in the definition of $KS(f)$. We have.

$$\xymatrix{
	M\ar[r]^h\ar[d]^i&M\ar[d]^i\ar[rdd]^{\hat{g}}&&\\
	T_{h^{-1}\circ f\circ h}\ar[r]^{h\times id}&T_f\ar[d]&&\\
	&W\ar[r]^{\bar{G}}&B\ar[r]^{\hat{h}}&B
	}$$ 
This implies 
\begin{prop}\label{conju}
	$KS(h^{-1}\circ f\circ h)=deg(f)KS(f)\circ\hat{h}^\star$.
\end{prop}

\begin{proof}
	Since $p=\hat{h}\circ p\circ h$, by definition we have 
	$$KS(f)=KS(deg(f)W,\hat{h}\circ\bar{G})$$
	$$=deg(f)KS(W,\bar{G})\hat{h}^\star$$
	$$=deg(f)KS(f)\circ\hat{h}^\star$$
\end{proof}

\subsection{Determination of $I(N_r)$}\label{subsec:det}
 In this subsection we use the generalized Kreck-Stolz invariant developed in the last two section to estimate a lower bound for the subgroup (of $I(N_r)$) generated by $g_1,g_2,g_3$ and an upperbound for the Torelli group when $N_r$ is spin. 
 	
Since we are considering spin case, by Proposition \ref{prop:real3} \ref{Hcp}, we may assume $r=1-4l$ and identify $N_r$ as $\mathbb{P}\gamma_{1,l}$. Let $\{s, t\}$ be its canonical generator (defined in Proposition \ref{Hcp}) and hence a polarization of $N_r$. Recall the definitions in chapter \ref{bri} we have:
$$B=B(2)=K(\z^2,2)\times BSpin,\ p=p_0=\oplus\circ(\star,p)$$
$$\hat{g}=(s,t,g_{Spin}):M\to K(\z^2,2)\times BSpin=B$$
	
Let $x,y\in H^2(B)=H^2(K(\z^2,2))=Hom(\z^2,\z)$ be the elements correspond to the projection to the first and second factor from $\z^2$ to $\z$. Then $H^2(B)=\mathbb{Z}\{x,y\}, H^4(B)=\mathbb{Z}\{x^2,xy,y^2,\hat{p}_1\}$, where $\hat{p}_1$ is the first spin Potrjagin class. The homomorphism $g^* \colon H^4(B) \to H^4(N_r)$ is given by  
		$$\hat{g}^* \colon H^4(B)\to H^4(N_r), (x^2,xy,y^2,\hat{p}_1)\to (s^2,st,st-ls^2,(2l-2)s^2),$$
and in $H^4(N_r)$ we have the relations (see Proposition \ref{Hcp})
$$\hat{g}^*(\hat{p}_1)=-\hat{p}_1(N_r)=(2l-2)s^2, \ \ t^2-ts+ls^2=0.$$
Therefore 
	$$K:=\mathrm{Ker}(\hat{g}^* \colon H^4(B)\to H^4(N_r))=\mathbb{Z}\{\alpha,\beta\}$$
where $\alpha:=y^2-xy+lx^2$,  $\beta:=\hat{p}_1+2(1-l)x^2$, and $Sym^2(K)$ has basis $$\sigma=(\sigma_1,\sigma_2,\sigma_3)=([\alpha\otimes\alpha],[\alpha\otimes\alpha]+[\alpha\otimes\beta],[\beta\otimes\beta]).$$ Thus we have the $KS$-invariant induced by this basis:
	$$KS_\sigma \colon I(N_r,D)\to \mathbb{Z}^3/L_\sigma$$
Which (by Proposition \ref{KS}) is a well defined injective homomorphism.

By applying Proposition \ref{conju}, we implies the third part of our main theorem:

\begin{prop}\label{rconju}
	When $r=1-4l$, the action of $MCG(N_r)/I(N_r)$ on $I(N_r)$ induced by conjugation is given by : $$[h^{-1}\circ f\circ h]=deg(h)[f],\ \ \forall h\in\m(N_r),f\in I(N_r).$$ 
\end{prop}
\begin{proof}
	By identifying $N_r$ as $\mathbb{P}\gamma_{1,l}$, we may identify $MCG(N_r)/I(N_r)$ as $A_{1,l}\subset GL_2(\z)$ via $R_{1,l}$. Then $h\in MCG(N_r)$ defines an automorphism $K(h,2)$ of $K(\z^2,2)$ such that $\pi_2K(h,2)=R_{1,l}(h)$ and the fiber homotopy equivalence $$\hat{h}=(K(h^{-1},2),id):B=K(\z^2,2)\times BSpin\to K(\z^2,2)\times BSpin=B$$ satisfying $\hat{g}=\hat{h}\circ \hat{g}\circ h$. By Proposition \ref{conju} we only need to compute the influence of $\hat{h}$ on $K$.
	
	By Lemma \ref{A+}, $A_{1,l}$ is ether generate by $X_1$ ( if $l\neq 1$ ) or $X_1$ and $H$ ( if $l=1$ ). 
	
	If $R_{1,l}(h^{-1})=X_1$, we have,
	$$\hat{h}^\star(x)=-x,\ \hat{h}^\star(y)=y-x,$$
	$$\hat{h}^\star(\alpha)=(y-x)^2-(-x)(y-x)+lx^2=y^2-xy+lx^2=\alpha$$
	$$\hat{h}^\star(\beta)=\hat{p}_1+2(1-l)(-x)^2=\hat{p}_1+2(1-l)x^2=\beta.$$
	
	If $l=1$ and $R_{1,1}(h^{-1})=X_1$, we have,
	$$\hat{h}^\star(x)=y,\ \hat{h}^\star(y)=x,$$
	$$\hat{h}^\star(\alpha)=(x)^2-(y)(x)+1y^2=y^2-xy+lx^2=\alpha$$
	$$\hat{h}^\star(\beta)=\hat{p}_1+2(1-1)y^2=\hat{p}_1+2(l-1)x^2=\beta.$$
	
	This implies $\hat{h}$ fix $\alpha,\beta$ and thus $K$. The Proposition is then implied from Proposition \ref{conju}.

\end{proof}

 Now we determine the lattice $L_\sigma$. The cup product defines an homomorphism $U:Sym^2K\to H^8(B)$ and a cohomology class $c\in H^8(B)$ defines a characteristic number $c:\Omega_8(B,p)\to \z$ takes $[M,f]$ to $\langle c,f_\star[M]\rangle$. These constructions satisfying
$$KS(X)(s)=U(s)(X),\ \forall X\in \Omega_8(B,p),\ s\in Sym^2K.$$
Thus $L_\sigma$ is exactly the lattice $$\{(U(\sigma_1),U(\sigma_2),U(\sigma_3))(X):X\in \Omega_8(B,p),Sing(X)=0\}.$$

To determine $L_\sigma$, we need to know $\Omega_8(B,p)=\Omega_8^{Spin}(K(\z^2,2))$. The cobordism group $\Omega_8(B,p)$ and the corresponding characteristic numbers are computed in \cite{Xu25}. We present the results in the following proposition. A sketch of the proof is given in the Appendix. 
	
	\begin{prop}\label{XFP}
		$\Omega_8(B,p)$ is a free $\mathbb{Z}$-module with basis $\{b_1,b_2...b_{10}\}$. The characteristic numbers are given in the following table.
\begin{center}		
		\begin{tabular}{|l|c|c|c|c|c|c|c|c|c|c|}
			\hline
		&$\hat{p}_1x^2$&$\hat{p}_1xy$&$\hat{p}_1y^2$&$x^4$&$x^3y$&$x^2y^2$&$xy^3$&$y^4$&$\hat{p}_1^2$&$Sign$\\
		\hline
		$b_1$&$24$&$0$&$0$&$0$&$0$&$0$&$0$&$0$&$0$&$0$\\
		$b_2$&$0$&$12$&$0$&$0$&$0$&$0$&$0$&$0$&$0$&$0$\\
		$b_3$&$0$&$0$&$24$&$0$&$0$&$0$&$0$&$0$&$0$&$0$\\
		$b_4$&$-2$&$0$&$0$&$2$&$0$&$0$&$0$&$0$&$0$&$0$\\
		$b_5$&$0$&$-2$&$0$&$0$&$1$&$0$&$0$&$0$&$0$&$0$\\
		$b_6$&$0$&$6$&$0$&$0$&$0$&$2$&$0$&$0$&$0$&$0$\\
		$b_7$&$0$&$-2$&$0$&$0$&$0$&$0$&$1$&$0$&$0$&$0$\\
		$b_8$&$0$&$0$&$-2$&$0$&$0$&$0$&$0$&$2$&$0$&$0$\\
		$b_9$&$0$&$0$&$0$&$0$&$0$&$0$&$0$&$0$&$2^57$&$0$\\
		$b_{10}$&$0$&$0$&$0$&$0$&$0$&$0$&$0$&$0$&$1$&$1$\\

		\hline
		\end{tabular}
\end{center}		
	\end{prop} 
	
To compute the range $\z^3 /L_\sigma$ of the $KS$-invariant. First of all, 
from the information given in the above proposition,  by a direct calculation one obtains the following table. 
\begin{center}	
	\begin{tabular}{|l|c|c|c|}
		\hline
		&$U(\sigma_1)$&$U(\sigma_2)$&$U(\sigma_3)$\\
		\hline
		$b_1$&$0$&$24l$&$2^53(1-l)$\\
		$b_2$&$0$&$12$&$0$\\
		$b_3$&$0$&$24$&$0$\\
		$b_4$&$2l^2$&$2l-2l^2$&$2^3(l^2-l)$\\
		$b_5$&$-2l$&$0$&$0$\\
		$b_6$&$4l+2$&$0$&$0$\\
		$b_7$&$-2$&$0$&$0$\\
		$b_8$&$2$&$0$&$0$\\
		$b_9$&$0$&$0$&$2^57$\\
		\hline
	\end{tabular}
\end{center}	 
One may simplify the matrix  by elementary row operations and shows that the lattice $L_\sigma$ is generated  by the row vectors of the following matrix.
\begin{center}
	  \begin{tabular}{|l|c|c|c|}
	 	\hline
	 	&$U(\sigma_1)$&$U(\sigma_2)$&$U(\sigma_3)$\\
	 	\hline	 	
	 	$X_1$&$2$&$0$&$0$\\	 	
	 	$X_2$&$0$&$12$&$0$\\
	 	$X_3$&$0$&$2l-2l^2$&$2^3(l^2-l)$\\
	 	$X_4$&$0$&$0$&$2^53(1-l)$\\
        $X_5$&$0$&$0$&$2^57$\\
	 	\hline
	 \end{tabular}
\end{center}	 
To make a further simplification, let $A=\mathrm{gcd}(6,l^2-l)$ and suppose 
$$A=6B-(l^2-l)C, \ \ 6=AD.$$ 
Then we have an invertible matrix
	$$\begin{pmatrix}
		B&C\\(l^2-l)/A&D
	\end{pmatrix}$$
The lattice $L_\sigma$ is generated by the row vectors of the following table
\begin{center}	
	 \begin{tabular}{|l|c|c|c|}
		\hline
		&$U(\sigma_1)$&$U(\sigma_2)$&$U(\sigma_3)$\\
		\hline	 	
		$X_1$&$2$&$0$&$0$\\	 	
		$BX_2+CX_3$&$0$&$2A$&$8C(l^2-l)$\\
		$(l^2-l)/AX_2+DX_3$&$0$&$0$&$8D(l^2-l)$\\
		$X_4$&$0$&$0$&$2^53(1-l)$\\
		$X_5$&$0$&$0$&$2^57$\\
		\hline
	\end{tabular}
\end{center}	
Since $D\in\{1,3\}$ we have $\mathrm{gcd}(8D(l^2-l),2^53(1-l),2^57)=8\mathrm{gcd}(28,4l-4,l^2-l)$ is a factor of $8C(l^2-l)$. This gives a further reduction. The lattice $L$ is generated by the row vectors of the following table
\begin{center}	
	\begin{tabular}{|l|c|c|c|}
		\hline
		&$U(\sigma_1)$&$U(\sigma_2)$&$U(\sigma_3)$\\
		\hline	 	
		$X_1$&$2$&$0$&$0$\\	 	
		$Y_2$&$0$&$2\mathrm{gcd}(6,l^2-l)$&$0$\\
		$Y_3$&$0$&$0$&$8\mathrm{gcd}(28,4l-4,l^2-l)$\\
		\hline
	\end{tabular}
\end{center}		
	
\begin{prop}
 $L_\sigma=2\mathbb{Z}\oplus2\mathrm{gcd}(6,l^2-l)\mathbb{Z}\oplus8\mathrm{gcd}(28,4l-4,l^2-l)\mathbb{Z}$. 
\end{prop}

The final step in the determination of $I(N_r)$ is to compute the orders of the generators $g_1$ and $g_2$, via $KS_\sigma$. Before doing this, let us recall how \cite[Theorem 7.9 and 7.12]{KreSu} detect $g_1, g_2, g_3\in\m_\partial(S^2\times D^4)$ using $KS$-invariant on $I(S^2\times S^4,D)$.

By Remark \ref{S2S4}, consider $N=S^2\times S^4, (B',p')=(B(1),p_0)$. Then $H^4(B')=\mathbb{Z}\{\alpha',\}$, where $\beta'$ is the first spin Potrjagin class and $\alpha'$ is the square of a generator of $H^2(B')$. Direct computation shows$$K':=Ker(\hat{g}^\star:H^4(B')\to H^4(S^2\times S^4))=H^4(B').$$ thus $Sym^2(K')$ has basis $$\sigma'=(\sigma'_1,\sigma'_2,\sigma'_3):=([\alpha'\otimes\alpha'],[\alpha'\otimes\alpha']+[\alpha'\otimes\beta'],[\beta'\otimes\beta']).$$ Thus we have the $KS$-invariant induced by this basis:
$$KS_\sigma' \colon I(S^2\times S^4,D)\to \mathbb{Z}^3/L'_\sigma$$

Using the natural decomposition $S^2 \times S^4 = (S^2 \times D^4)_+ \cup (S^2 \times D^4)_-$, one has a homomorphism  $\m_{\partial}(S^2 \times D^4) \to I(S^2 \times S^4, D)$ defined by ``extension by the identity on $(S^2 \times D^4)_-$", where $I(S^2 \times S^4,D)$ is the Torreli group of $S^2 \times S^4$ relative to a fixed $6$-dimensional disk in  $(S^2 \times D^4)_-$. Thus we have:
$$\xymatrix{E':\m_{\partial}(S^2 \times D^4)\ar[r]&I(S^2 \times S^4, D)\ar[r]^{\ \ KS_\sigma'}&\mathbb{Z}^3/L'_\sigma}$$

It is shown in \cite[Theorem 6.2, 7.9]{KreSu} that the ``extension by the identity on $(S^2 \times D^4)_-$" map induce an isomorphism $\m_{\partial}(S^2 \times D^4) \cong I(S^2 \times S^4, D)$ and
$$L'_\sigma=2\mathbb{Z}\oplus2\cdot12\mathbb{Z}\oplus8\cdot28\mathbb{Z},$$
$$E'(ag_3+bg_2+cg_1)=[(a,2b,8c)].$$

Let $S^2 \times D^4 \subset \mathbb{P}(\gamma_{1,l})=N_r$ be a local trivialization which has no intersection with the fixed disk $D$. It induces a homomorphism 	$\m_\partial(S^2\times D^4)\to I(N_r,D)$. Denote the composition of this homomorphism and $KS_\sigma$ by
$$E \colon \m_\partial(S^2\times D^4) \to \z^3/L_\sigma.$$

\begin{prop}\label{hom}
		$$E(ag_3+bg_2+cg_1)=[(a,2b,8c)]\in \mathbb{Z}_2\oplus\mathbb{Z}_{2\mathrm{gcd}(6,l^2-l)}\oplus\mathbb{Z}_{8\mathrm{gcd}(28,4l-4,l^2-l)}. $$
	\end{prop}
	\begin{proof}
		The statement of the proposition follows from the following commutative diagram 
        $$\xymatrix{\m_{\partial}(S^2 \times D^4)\ar[r]^{E'}\ar[d]^e&\mathbb{Z}_2\oplus2\mathbb{Z}_{2\cdot12}\oplus8\mathbb{Z}_{8\cdot28}\ar[d]^{\mathrm{id}_{\mathbb{Z}^3}}\\I(N_r,D)\ar[r]^{KS_\sigma \ \ \ \ \ \ \ \ \ \ \ \ \ \ \ \ }&\mathbb{Z}_2\oplus\mathbb{Z}_{2\mathrm{gcd}(6,l^2-l)}\oplus\mathbb{Z}_{8\mathrm{gcd}(28,4l-4,l^2-l)}
        }$$
where $e$ is the ``extension by the identity" homomorphism through a fixed local trivalization $S^2\times D^4\subset N_r$; the right hand side vertical map is induced by the identity of $\z^3$, and is indeed the quotient map on each direct summand.
        
Now we give a proof of the commutativity of the diagram. Let $i_B \colon B'=\{*\}\times K(\z,2)\times BSpin\to K(\z,2)\times K(\z,2)\times BSpin=B$ be the inclusion map. Since it is compatible with the projections $p$ and $p'$,  $i_{B*}$ takes a $B'$-manifold to a $B$-manifold. Given $f\in I(S^2\times S^4,D)$, since it is in the image of $\m_{\partial}(S^2 \times D^4)$, we may assume that $f$ is the identity on $(S^2 \times D^4)_-$. Let $(V,G_V)$ be the $B'$-coboundary needed in computing $KS_\sigma'(f)$, identify $(S^2\times D^4)_-\times S^1\subset T_f=\partial V$ with $S^2\times D^4\times S^1\in \partial(N_r\times D^2)$ and obtain a coboundary $W=(N_r\times D^2)\cup_{S^2\times D^4\times S^1}V$. We have $\partial W=T_{e(f)}$ and similar argument analog to the proof of Proposition \ref{KS} shows that there is a normal $B$-structure $G_w$ on $W$ extending $i_B(V,G_v)$ and $(N_r,\hat g)\times D^2$. Also similar to Proposition \ref{KS} the $(N_r,\hat g)\times D^2$ part doesn't contribute when we computing $KS_\sigma(e(f))$ using $(W,G_w)$. Since $\alpha'=i_B^*\alpha, \beta'=i_B^*\beta$
        we have $KS_\sigma(e(f))=KS_\sigma'(f)$ and the proposition is proved.
	\end{proof}	

The proof of Theorem \ref{Mth} is completed by the following propositions.	

\begin{prop}
	For $r=1-4l$, $$\mathrm{gcd}(6,l^2-l)\mathrm{gcd}(28,l^2-l,4l-4)\leq |I(N_r)|\leq32\mathrm{gcd}(6,l^2-l)\mathrm{gcd}(28,l^2-l,4l-4).$$ 
\end{prop}
\begin{proof}
	The above proposition shows $$16^{-1}|\z^3/L_\sigma|\leq|I(N_r,D)|\leq|\z^3/L_\sigma|=32\mathrm{gcd}(6,l^2-l)\mathrm{gcd}(28,l^2-l,4l-4).$$
	By \cite{Wa63} the kernel of $I(N_r,D)\to I(N_r)$ is isomorphic to $\{0\}$ or $\mathbb{Z}_2$, thus
	$$2^{-1}|I(N_r,D)|\leq I(N_r)\leq|I(N_r,D)|.$$
	This prove the proposition.
\end{proof}
	
    \begin{prop}
	For $r=8n+5$,
		$$I(N_r)=\mathbb{Z}_{2\mathrm{\mathrm{gcd}}(3,2n^2+3n+1)}\{g_2\}\oplus \mathbb{Z}_{2\mathrm{gcd}(14,n+1)}\{g_1\}$$ 
	\end{prop}
	\begin{proof}
		Since $N_r=\mathbb{P}\gamma_{1,l}$ we have $l=-2n-1$. Therefore 
		$$\mathrm{gcd}(6,l^2-l)=2\mathrm{gcd}(3,2n^2+3n+1)$$
		 $$\mathrm{gcd}(28,4l-4,l^2-l)=\mathrm{gcd}(28,8(n+1),2(2n+1)(n+1))=2\mathrm{gcd}(14,n+1).$$
By Proposition \ref{KS} and \ref{hom}, $I(N_r,D)$ contains a subgroup 
$$\mathbb{Z}_2\{g_3\}\oplus\mathbb{Z}_{2\mathrm{gcd}(3,2n^2+3n+1)}\{g_2\}\oplus \mathbb{Z}_{2\mathrm{gcd}(14,n+1)}\{g_1\}.$$ By Proposition \ref{sur}, the subgroup $\mathbb{Z}_{2\mathrm{gcd}(3,2n^2+3n+1)}\{g_2\}\oplus \mathbb{Z}_{2\mathrm{gcd}(14,n+1)}\{g_1\}$ is mapped onto $I(N_r)$. By \cite{Wa63} the kernel of $I(N_r,D)\to I(N_r)$ is isomorphic to $\{0\}$ or $\mathbb{Z}_2$. Hence $\mathbb{Z}_{2\mathrm{gcd}(3,2n^2+3n+1)}\{g_2\}\oplus \mathbb{Z}_{2\mathrm{gcd}(14,n+1)}\{g_1\}\to I(N_r)$ is an isomorphism. 
	\end{proof}
	
	\begin{cor}
		For $r=8n+5$,
		$$I^h(N_r)=\mathbb{Z}_{2\mathrm{\mathrm{gcd}}(3,2n^2+3n+1)}\{g_2\}$$
	\end{cor}
	\begin{proof}
		By Proposition \ref{sur}, Lemma \ref{t7} and the exact sequence before lemma \ref{t7} $I^h(N_r)\cong I(N_r)/\langle g_1\rangle$.
	\end{proof}
	
	\section{Appendix}
	
The bordism group $\Omega_8(B(2),p_0)=\Omega_8^{Spin}(K(\z^2,2))$ was computed in \cite{Xu25}. In this appendix we give another proof avoiding the use of Adams spectral sequence.

Let's first recall some known results of the spin bordism group $\Omega_*^{Spin}$ \cite[p.92]{LaMi}:
\begin{center}		
	\begin{tabular}{|l|c|c|c|c|c|c|c|c|c|}
		\hline
		$i$&$0$&$1$&$2$&$3$&$4$&$5$&$6$&$7$&$8$\\
		\hline
		$\Omega_i^{Spin}$&$\z$&$\z_2$&$\z_2$&$0$&$\z$&$0$&$0$&$0$&$\z^2$\\
		\hline
	\end{tabular}
\end{center}
generators of $\Omega_4^{Spin}$	and $\Omega_8^{Spin}$ are also given in \cite[p.92]{LaMi}, they can be detected by cohomology classes of $BSpin$ and signature:
$$\xymatrix{\hat{p}_1:\Omega_4^{Spin}\ar[r]^{\ \ \ \cong}&24\z}$$
$$\xymatrix{(\hat{p}_1^2,Sign-\hat{p}_1^2):\Omega_8^{Spin}\ar[r]^{\ \ \ \ \ \ \ \ \ \cong}&\z\oplus224\z}$$

Since the Atiyah-Hirzebruch spectral sequence of the functor $\Omega_\star^{Spin}(-):Top\to Ab_\star$ is natural, the $E^2$-page differentials (of AHSS) are natural transformations $d^2_{p,q}:H_p(-,\Omega_q^{Spin})\to H_{p-2}(-,\Omega_{q+1}^{Spin})$. And \cite[p.273]{Zh01} shows: $$d^2_{p,0}:H_8(K(\z^2,2),\Omega_0^{Spin})\to H_{p-2}(K(\z^2,2),\Omega_1^{Spin})$$
$$d^2_{p,1}:H_8(K(\z^2,2),\Omega_1^{Spin})\to H_{p-2}(K(\z^2,2),\Omega_2^{Spin})$$
can be identify as:
$$\xymatrix{d^2_{8,0}:H_8(K(\z^2,2),\mathbb{Z})\ar[r]^{\rho_2}&H_8(K(\z^2,2),\mathbb{Z}_2)\ar[r]^{Sq^{2*}}&H_6(K(\z^2,2),,\mathbb{Z}_2)}$$
$$\xymatrix{d^2_{8,1}:H_8(K(\z^2,2),\mathbb{Z}_2)\ar[r]^{Sq^{2*}}&H_6(K(\z^2,2),\mathbb{Z}_2)}$$
Here $Sq^{2*}$ is the dual of the Steenrod square $Sq^2$ and $\rho_2$ is the "reduction modulo 2".

	\begin{proof}[Proof of Proposition \ref{XFP}]
The bordism group $\Omega_8(B,p)$ is the spin bordism group $\Omega_8^{Spin}(K(\z^2,2))$. The Atiyah-Hirzebruch spectral sequence for this generalized homology group has $E^2$-terms  $E_{p,q}^2=H_p(K(\z^2,2),\Omega_q^{Spin})$, thus the $E^2$-page looks like 
\begin{center}		
		\begin{tabular}{|l|c|c|c|c|c|c|c|c|c|c|}
			$8$&$\mathbb{Z}^2$&$0$&$\mathbb{Z}^4$&$0$&$\mathbb{Z}^6$&$0$&$\mathbb{Z}^8$&$0$&$\mathbb{Z}^{10}$&$0$\\
			$7$&$0$&$0$&$0$&$0$&$0$&$0$&$0$&$0$&$0$&$0$\\
			$6$&$0$&$0$&$0$&$0$&$0$&$0$&$0$&$0$&$0$&$0$\\
			$5$&$0$&$0$&$0$&$0$&$0$&$0$&$0$&$0$&$0$&$0$\\
			$4$&$\mathbb{Z}$&$0$&$\mathbb{Z}^2$&$0$&$\mathbb{Z}^3$&$0$&$\mathbb{Z}^4$&$0$&$\mathbb{Z}^5$&$0$\\
			$3$&$0$&$0$&$0$&$0$&$0$&$0$&$0$&$0$&$0$&$0$\\
			$2$&$\mathbb{Z}_2$&$0$&$\mathbb{Z}_2^2$&$0$&$\mathbb{Z}_2^3$&$0$&$\mathbb{Z}_2^4$&$0$&$\mathbb{Z}_2^5$&$0$\\
		    $1$&$\mathbb{Z}_2$&$0$&$\mathbb{Z}_2^2$&$0$&$\mathbb{Z}_2^3$&$0$&$\mathbb{Z}_2^4$&$0$&$\mathbb{Z}_2^5$&$0$\\
			$0$&$\mathbb{Z}$&$0$&$\mathbb{Z}^2$&$0$&$\mathbb{Z}^3$&$0$&$\mathbb{Z}^4$&$0$&$\mathbb{Z}^5$&$0$\\
			\hline
			&$0$&$1$&$2$&$3$&$4$&$5$&$6$&$7$&$8$&$9$\\
			
			\hline
		\end{tabular}
\end{center}
		
The coefficient group $\Omega_8^{Spin}\to \Omega_8^{Spin}(K(\z^2,2))$ is a direct summand. One only needs to analyze the following differentials
		$$d^2_{8,0}:H_8(K(\z^2,2),\Omega_0^{Spin})\to H_6(K(\z^2,2),\Omega_1^{Spin})$$
		$$d^2_{8,1}:H_8(K(\z^2,2),\Omega_1^{Spin})\to H_6(K(\z^2,2),\Omega_2^{Spin})$$
		By \cite{Zh01}, they identify as $d^2_{8,0}=Sq^{2*}\circ \rho_2$ and $d^2_{6,0}=Sq^{2*}$.
		
		By direct computation of the above differential, there is a filtration $\Omega_8^{Spin}=F_0\subset F_4\subset F_6\subset F_8=\Omega_8^{Spin}(K(\z^2,2))$ with quotient groups
		$$F_8/F_6=Kerd^2_{8,0}=\mathbb{Z}\{2x^{4*},x^3y^*,2x^2y^{2*},xy^{3*},2y^{4*}\}\cong \mathbb{Z}^5$$
		$$F_6/F_4=coKerd^2_{8,1}=\mathbb{Z}_2\{x^2y^*-xy^{2*}\}\cong\mathbb{Z}_2$$
		$$F_4/F_0=H^4(K(\z^2,2),\Omega_4^{Spin})\cong\mathbb{Z}^3$$
		$$F_0=\Omega_8^{Spin}=\mathbb{Z}\{\mathbb{HP}^2,Bott\}\cong\mathbb{Z}^2$$
		Here (by considering the "Hurewicz natural transformation" $\Omega_\star^{Spin}(-)\to H_\star(-)$) $F_8/F_6$ is viewed as the image of Hurewicz map $\Omega_8^{Spin}(K(\z^2,2))\to H_8(K(\z^2,2))$, and $\{x^iy^{n-i*}\}_{i=0}^n\in H_8(K(\z^2,2))$ is the dual basis of $\{x^iy^{n-i}\}_{i=0}^n$.
		
		View $K(\z^2,2)$ as CW-complex $\mathbb{CP}^\infty\times \mathbb{CP}^\infty$ (with the ordinary CW-structure) and let $(\mathbb{CP}^\infty\times \mathbb{CP}^\infty)_i$ be its $i$-skeleton, we have inclusions:
		$$\xymatrix{
			((\mathbb{CP}^\infty\times \mathbb{CP}^\infty)_4,\emptyset)\ar[r]\ar[d]&(\mathbb{CP}^\infty\times \mathbb{CP}^\infty,\emptyset)\\
			((\mathbb{CP}^\infty\times \mathbb{CP}^\infty)_4,(\mathbb{CP}^\infty\times \mathbb{CP}^\infty)_2)&}$$
		
		By comparing Atiyah-Hirzebruch spectral sequence of the above pairs of spaces, these inclusion map induce isomorphism:
		
		$$F_6/F_4=\Omega_8^{Spin}((\mathbb{CP}^\infty\times \mathbb{CP}^\infty)_4,\mathbb{CP}^\infty\times \mathbb{CP}^\infty)_2)=\hat{\Omega}_8^{Spin}(\vee_{i=1}^3 S^4)=\oplus_{i=1}^3\Omega_4^{Spin}$$
		Combing the fact $\hat{p}_1:\Omega_4^{Spin}\cong 24\mathbb{Z}$, we have isomorphism:
		$$(\hat{p}_1x^2,\hat{p}_1xy,\hat{p}_1y^2): F_4/F_0\to 24\mathbb{Z}^3$$  
		
		Let $V^4_2\subset \mathbb{P}^5$ be a degree $2$ hypersurface, $c\in H^2(V^4_2)$ be the restriction of $c_1(O(-1))\in H^2(\mathbb{P}^5)$, $Sp_A$ be a spin structure of $V^4_2$. Define $M_4,M_6,M_8\in \Omega_8^{Spin}(K(\z^2,2))$ be $(V^4_2,c,0,Sp_A), (V^4_2,c,c,Sp_A),(V^4_2,0,c,Sp_A)$. 
		
		Let $a,b,\in H^2(\mathbb{P}\times \mathbb{P}^3)$ be $c_1(O(-1)\times \mathbb{C}),c_1(\mathbb{C}\times O(-1))$,$Sp_B$ be a spin structure of $\mathbb{P}\times \mathbb{P}^3$. Define $M_5,M_7\in\Omega_8^{Spin}(K(\z^2,2))$ be $(\mathbb{P}\times \mathbb{P}^3,a,b,Sp_B), (\mathbb{P}\times \mathbb{P}^3,b,a,Sp_B)$. 
		
		Let $\{d_i\}_i^4\in H^2(S^2\times S^2\times S^2\times S^2)$ be the dual basis of the fundamental class of the $S^2$ factors, $Sp_C$ be a spin structure of $S^2\times S^2\times S^2\times S^2$. Define $M_1,M_2, M_3\in\Omega_8^{Spin}(K(\z^2,2))$ be $(S^2\times S^2\times S^2\times S^2,d_1+d_2+d_3+d_4,0,Sp_C), (S^2\times S^2\times S^2\times S^2,d_1+d_2+d_3,d4,Sp_C), (S^2\times S^2\times S^2\times S^2,0,d_1+d_2+d_3+d_4,Sp_C)$.
		
		direct computation gives:
\begin{center}		
		\begin{tabular}{|l|c|c|c|c|c|c|c|c|c|c|}
			\hline
			&$\hat{p}_1x^2$&$\hat{p}_1xy$&$\hat{p}_1y^2$&$x^4$&$x^3y$&$x^2y^2$&$xy^3$&$y^4$&$\hat{p}_1^2$&$Sign$\\
			\hline
			$M_1$&$0$&$0$&$0$&$24$&$0$&$0$&$0$&$0$&$0$&$0$\\
			$M_2$&$0$&$0$&$0$&$0$&$6$&$0$&$0$&$0$&$0$&$0$\\
			$M_3$&$0$&$0$&$0$&$0$&$0$&$0$&$0$&$24$&$0$&$0$\\
			$M_4$&$-2$&$0$&$0$&$2$&$0$&$0$&$0$&$0$&$2$&$2$\\
			$M_5$&$0$&$-2$&$0$&$0$&$1$&$0$&$0$&$0$&$0$&$0$\\
			$M_6$&$-2$&$-2$&$-2$&$2$&$2$&$2$&$2$&$2$&$2$&$2$\\
			$M_7$&$0$&$-2$&$0$&$0$&$0$&$0$&$1$&$0$&$0$&$0$\\
			$M_8$&$0$&$0$&$-2$&$0$&$0$&$0$&$0$&$2$&$2$&$2$\\
			$\mathbb{HP}$&$0$&$0$&$0$&$0$&$0$&$0$&$0$&$0$&$1$&$1$\\
			$Bott$&$0$&$0$&$0$&$0$&$0$&$0$&$0$&$0$&$0$&$2^57$\\

			\hline
		\end{tabular}
\end{center}		
		By elementary row operations, we obtain another sequence of elements in $\Omega_8^{Spin}(K(\z^2,2))$, and this tabular can be change to
\begin{center}		
		\begin{tabular}{|l|c|c|c|c|c|c|c|c|c|c|}
			\hline
			&$\hat{p}_1x^2$&$\hat{p}_1xy$&$\hat{p}_1y^2$&$x^4$&$x^3y$&$x^2y^2$&$xy^3$&$y^4$&$\hat{p}_1^2$&$Sign$\\
			\hline
			$b_1$&$24$&$0$&$0$&$0$&$0$&$0$&$0$&$0$&$0$&$0$\\
			$b_2$&$0$&$12$&$0$&$0$&$0$&$0$&$0$&$0$&$0$&$0$\\
			$b_3$&$0$&$0$&$24$&$0$&$0$&$0$&$0$&$0$&$0$&$0$\\
			$b_4$&$-2$&$0$&$0$&$2$&$0$&$0$&$0$&$0$&$0$&$0$\\
			$b_5$&$0$&$-2$&$0$&$0$&$1$&$0$&$0$&$0$&$0$&$0$\\
			$b_6$&$0$&$6$&$0$&$0$&$0$&$2$&$0$&$0$&$0$&$0$\\
			$b_7$&$0$&$-2$&$0$&$0$&$0$&$0$&$1$&$0$&$0$&$0$\\
			$b_8$&$0$&$0$&$-2$&$0$&$0$&$0$&$0$&$2$&$0$&$0$\\
			$\mathbb{HP}$&$0$&$0$&$0$&$0$&$0$&$0$&$0$&$0$&$1$&$1$\\
			$Bott$&$0$&$0$&$0$&$0$&$0$&$0$&$0$&$0$&$0$&$2^57$\\
			
			\hline
		\end{tabular}
\end{center}			
		Recall that we have:
		$$F_8/F_6=\mathbb{Z}\{2x^{4*},x^3y^*,2x^2y^{2*},xy^{3*},2y^{4*}\}$$
		$$F_6/F_4\cong\mathbb{Z}_2$$
		$$(\hat{p}_1x^2,\hat{p}_1xy,\hat{p}_1y^2): F_4/F_0\to 24\mathbb{Z}^3$$
		$$F_0=\mathbb{Z}\{\mathbb{HP}^2,Bott\}$$
		So the exact sequence $0\to F_4/F_0\to F_6/F_0\to F_6/F_4\to 0$ is non-split, and $\{b_i\}_{i=1}^8\cup \{\mathbb{HP}^2,Bott\}$ is a basis of $\Omega_8^{Spin}(K(\z^2,2))$, this implies Proposition \ref{XFP}.
	\end{proof}

\end{document}